\documentclass{compositio}
\usepackage{mathpazo}
\usepackage{lmodern}

\usepackage[inline]{enumitem}

\usepackage{amssymb,amsmath,latexsym,lipsum}
\usepackage{enumitem}
\usepackage{tabulary}
\usepackage{booktabs}
\usepackage{charter}
\usepackage[title]{appendix}
\usepackage{longtable}
\usepackage{amsthm}
\usepackage{euscript}
\usepackage{color}
\usepackage{float}
\usepackage[T1]{fontenc}
\usepackage[singlespacing]{setspace}
\usepackage[colorinlistoftodos]{todonotes}
\usepackage{tikz}
\usetikzlibrary {positioning}
\definecolor {processblue}{cmyk}{0.96,0,0,0}
\usepackage{bookmark}
\usepackage{multirow,bigdelim}
\usepackage{mathrsfs}
\usepackage{natbib}
\usepackage{mathtools}
\usepackage{tikz-cd}
\usepackage{pgfplots}
\usepackage{pgf}
\usetikzlibrary{arrows}

\pgfplotsset{compat=1.16}
\renewcommand{\contentsname}{Contents\\{\footnotesize\normalfont}}

\newtheorem{theorem}{Theorem}[section]

\newtheorem{lemma}[theorem]{Lemma}
\newtheorem{corollary}[theorem]{Corollary}
\newtheorem{definition}[theorem]{Definition}
\theoremstyle{remark}

\theoremstyle{remark}
\newtheorem{remark}[theorem]{Remark}

\DeclareMathOperator{\rc}{\mathfrak{tc}}
\newcommand{\Pj}{\mathbb{P}^{1}}
\newcommand{\PjS}{\mathbb{P}_S^{1}}

\newcommand{\Pjk}{\mathbb{P}_k^{1}}

\newcommand{\Mgnbar}{\overline{\mathcal{M}}_{g,n}}

\newcommand{\Mn}{\mathcal{M}_{0,n}}

\newcommand{\A}{\mathbb{A}}
\newcommand{\Div}{\mathcal{DIV}_{g,a}}
\newcommand{\Mnexm}{\mathcal{M}_{0,n}^{ex, \mathbf{m}}}
\newcommand{\Mnquexm}{\mathcal{M}_{0,n}^{qu-ex, \mathbf{m}}}
\newcommand{\AScovhen}{\mathcal{AS}cov_h^{\vec{e},n}}
\newcommand{\G}{\mathbb{G}}

\newcommand{\PLHURphgnZ}{\mathcal{PLH}_{A}^{\mathbb{Z}_{(p)}}}
\newcommand{\PLHURphgnF}{\mathcal{PLH}_{A, \Xi}^{\mathbb{F}_p}}
\newcommand{\LHURphgnZ}{\mathcal{LH}_{A}^{\mathbb{Z}_{(p)}}}
\newcommand{\LHURphgnF}{\mathcal{LH}_{A,\Xi}^{\mathbb{F}_p}}
\newcommand{\RubhbX}{\mathcal{RUB}^\Xi_{h,b}}

\newcommand{\Maphgnb}{\mathcal{M}ap^B_A}
\newcommand{\Mapfshgn}{\mathcal{M}ap_A^{B, fin-sur}}

\newcommand{\Z}{\mathbb{Z}}

\newcommand{\N}{\mathbb{N}}

\begin{document}
\newcommand{\Mggpn}{\mathcal{M}_{g,g'}[\mathbb{Z}/p^n]}
\newcommand{\Mggp}{\mathcal{M}_{g,g'}[\mathbb{Z}/p]}
\newcommand{\MggpG}{\mathcal{M}_{g,g'}[G]}
\newcommand{\Mggpnbar}{\overline{\mathcal{M}}_{g,g'}[\mathbb{Z/p}^n]}
\newcommand{\Mgg}{\mathcal{M}_{g,g'}}
\newcommand{\MggO}{\mathcal{M}_{g,g'=0}}
\newcommand{\MgnG}{\mathcal{M}_{g,n}[G]}
\newcommand{\OP}{\mathcal{O}_{\Pj}}
\title{Logarithmic Hurwitz Spaces in Mixed and Positive Characteristic with Wild Ramification}

\author{Matthias Hippold}
\email{matthias.hippold@mail.huji.ac.il}
\address{Einstein Institute of Mathematics, The Hebrew University of Jerusalem, Edmond J.
Safra Campus, Giv’at Ram, Jerusalem, 91904, Israel }

%\author{Michael Temkin}
%\email{michael.temkin@mail.huji.ac.il}
%\address{Einstein Institute of Mathematics, The Hebrew University of Jerusalem, Edmond J.
%Safra Campus, Giv’at Ram, Jerusalem, 91904, Israel}

%\dedication{A dedication can be included here.}
%\classification{14H30 (primary), 14H10, (secondary).}
%\keywords{Artin--Schreier-Witt theory, moduli space, connectedness.}
%\thanks{This file documents \pkg{compositio} version \Fileversion\ and
%was last revised \Filedate.}

\begin{abstract}
We introduce new logarithmic Hurwitz spaces $\LHURphgnZ$ and $\LHURphgnF$ over $\mathbb{Z}_{(p)}$ and $\mathbb{F}_p$ respectively that in the mixed characteristic case can be considered as a compactification of the admissible cover stack parametrizing ramified covers of curves in characteristic $0$ of degree $p$ and in the equicharacteristic case compactify the space of separable maps between smooth curves of degree $p$. These Hurwitz spaces will carry a logarithmic structure and to emphasize that they are informative, we prove that in some first cases our Hurwitz spaces are log smooth. To achieve this, we introduce various Moduli spaces that parametrize Artin--Schreier covers and the locus of zeroes and poles of certain differential forms, show their smoothness and compute their dimension.
%Futhermore, we will show that there exists a smooth filtration of the boundary similar to that of $\Mgnbar$.
\end{abstract}
 
\maketitle
\newtheorem*{theorem*}{\normalfont\scshape Theorem}
\vspace*{6pt}\tableofcontents
\section{Introduction}
When considering a finite surjective degree $p$ morphism of nodal curves over $\Z_{(p)}$, one notes that the special fiber behaves very differently compared to the generic fiber: While in characteristic $0$ at each irreducible component of the source curve the cover is given by a monodromy representation, this is far from being true in characteristic $p$. Here, at some components the cover is separable but it might as well be the case that at some components the cover is inseparable so is given by the relative Frobenius. 

Because of these complications, large parts of the mathematical community have excluded the case of degree $p$ covers in characteristic $p$ when studying Hurwitz spaces that are Moduli spaces that parametrize finite covers of curves. However, results from \cite{brezner2017lifting} provide criteria for which covers of degree $p$ can be lifted from characteristic $p$ to characteristic $0$. After fixing the data of $A=(h,g,N, \Lambda)$, we introduce $\LHURphgnZ$ as the closure of the Hurwitz space in characteristic $0$ of covers of degree $p$ of smooth curves of genus $g$ by a curve of genus $h$ with fixed ramification pattern $\Lambda$ in a large Moduli space of all finite maps of curves over $\Z_{(p)}$. We furthermore give an explicit modular description of this space that keeps track of the lifting data that ensures which covers in characteristic $p$ lift to characteristic $0$. This data consists of a piecewise linear function on the tropicalization of the source curve and bivariant logarithmic differential forms satisfying some conditions. Our main result will be the following theorem:
\begin{theorem*} 
    All the irreducible components of the special fiber of the logarithmic Hurwitz space $\LHURphgnZ$ are of dimension $3g-3+N$. For $p=2$ and $g=0$, $\mathcal{LH}_{A}^{\mathbb{Z}_{(2)}}$ is log smooth over $\Z_{(2)}$ with logarithmic structure coming from $\N \to \mathbb{Z}_{(2)}$ that sends $1$ to $2$.
\end{theorem*}
Similarly, for $A$ as above and $\Xi$ carrying additional information about the ramification profile $(\Lambda, \Xi)$, we also introduce the Hurwitz space $\LHURphgnF$ that compactifies the space of separable degree $p$ covers between smooth curves with fixed ramification profile in characteristic $p$ and show the following analogous result:
\begin{theorem*} 
    For $p=2$ and $g=0$, the Hurwitz space $\mathcal{LH}_{A, \Xi}^{\mathbb{F}_{2}}$ is log smooth over $\mathbb{F}_2$ with trivial logarithmic structure.
\end{theorem*}
We expect that by extending our findings in Section~\ref{sec:rel_obj} to curves of higher genus and separable but non-normal covers, we will be able to generalize the latter theorems to higher genus and degree as long as the degree is not divided by $p^2$ in future work.

The author wants to express his gratitude to his academic advisor Michael Temkin who suggested this problem as part of the author's Ph.D. project, to Dan Abramovich for providing valuable feedback and to his Ph.D. brother Simon Stojkovic for many helpful discussions about logarithmic geometry. 

This research was supported by ERC Consolidator Grant 770922 - BirNonArchGeom.
\subsection{Outline}
In Section~\ref{sec:preliminaries}, we will recall basic notions of logarithmic geometry with a focus on log curves and basic facts about the Hurwitz space in characteristic $0$ together with the minimal logarithmic structure on this Moduli space making it log smooth over its base field. We will also study the stack $\RubhbX$ that parametrizes piecewise linear functions on the dual graph of log curves. Finally, we will introduce exact and quasi-exact relative differential forms on curves and by using the Cartier operator describe them for the relative Frobenius more explicitly than in previous works.

Next, in Section~\ref{sec:mod_space} we will construct the Moduli spaces $\LHURphgnZ$ and $\LHURphgnF$ which are the main object of interest in this paper and study some of their first properties. Deformations of objects inside a fixed stratum of $\LHURphgnZ$ and $\LHURphgnF$ can be described as a product of deformations of differential forms and separable covers. Therefore, in Section~\ref{sec:rel_obj} we will study Moduli spaces of these objects. In Section~\ref{sec:log_smooth}, we use insights of the previous sections to prove our main theorem stating that in some first cases $\LHURphgnZ$ and $\LHURphgnF$ are log smooth. Finally, in the last section, Section~\ref{sec:example} we will illustrate our findings in an example.

\section{Preliminaries} \label{sec:preliminaries}
\subsection{Log Curves}
In this section, we will briefly recall basic notions of logarithmic geometry with a focus on curves and their Moduli space. Logarithmic structures were introduced in \cite{kato1989logarithmic}, a broader introduction is given in the survey paper \cite{abramovich2010logarithmic}. To our knowledge, the most extensive discussion of basic notions can be found in the textbook \cite{ogus2018lectures}.

\begin{definition}
    Let $(X, \mathcal{O}_X)$ be a scheme, $M$ be a sheaf of monoids in the étale topology on $X$ and $\alpha: M \to \mathcal{O}_X$ be a morphism of sheaves of monoids. If the restriction of $\alpha$ to the subsheaf of invertible elements is an isomorphism onto its image, then $\alpha$ is called a logarithmic structure on $M$. 

    A morphism between two logarithmic structures $M,N$ on $(X, \mathcal{O}_X)$ is a morphism of sheaves $M \to N$ that is compatible with the morphisms to $\mathcal{O}_X$. A morphism between logarithmic schemes $(X,\mathcal{O}_X, M_X)$ and $(Y, \mathcal{O}_Y, M_Y)$ is a morphism $f$ between the underlying schemes together with a morphism $f^*M_Y \to M_X$ of logarithmic structures on $X$.

    When the sheaf of regular functions and the logarithmic structure are clear from context we will sometimes abbreviate the notation by only $X$ instead of $(X,\mathcal{O}_X, M_X)$ or $(X,\mathcal{O}_X)$. 
\end{definition}
An important class of logarithmic schemes can be constructed from schemes with a divisor:
\begin{definition}
    Let $X$ be a regular scheme and $D$ be a normal crossing divisor on $X$. The logarithmic structure associated to $D$ consists of those elements of $\mathcal{O}_X$ that are invertible outside of $D$. 
\end{definition}
For a monoid $M$, its groupification $M^{gp}$ is defined to be the abelian group with the universal property such that any morphism from $M$ to an abelian group factors uniquely through $M^{gp}$ \citep[I.1.3.]{ogus2018lectures}. If the morphism $M \to M^{gp}$ is injective, we call $M$ integral. When there is a surjection $\N^n \to M$, the monoid $M$ is called finitely generated, if it is both integral and finitely generated, M is called fine. If for $m \in M^{gp}$ and $n\in \N$ we can conclude from $nm \in M$ that also $m \in M$, we say that $M$ is saturated. For $M$ being a logarithmic structure on $X$, we can form the characteristic monoid $\overline{M} \coloneqq M/\mathcal{O}_X^*$. The characteristic monoid can be thought to remember the combinatorial information coming from $M$.

In the following definition, we make use of the process of logarithmification to have enough units, details can be found for example in \citep[III.1.1.3]{ogus2018lectures}.

\begin{definition}
    For $M$ being a monoid, by $\mathbb{A}[M]$ we denote the logarithmic scheme with underlying scheme being $Spec(\Z[M])$ with logarithmic structure coming from $M \to \Z[M]$. A logarithmic scheme $X$ is called \emph{quasi-coherent} if at each point $x$ there exists an étale neighborhood of $x$ where the logarithmic scheme is isomorphic to $\mathbb{A}[M]$ for some monoid $M$. If this is the case, $M$ is called a \emph{chart} of $X$ at a point $x$.
\end{definition}

We now have the tools to introduce logarithmic curves. Instead of the more conceptual description that can be found in the literature discussed above, we will use the following ad hoc definition.

\begin{definition}
    Let $S$ be a logarithmic scheme. A logarithmic curve or log curve over $S$ is a morphism $\pi_C: C \to S$ of logarithmic schemes such that the underlying morphism of schemes is a nodal curve relative to $S$. This means that the underlying morphism is proper, surjective, finitely presented and the geometric fibers are nodal, projective and connected curves. Furthermore, regarding the the logarithmic structure we require that for each geometric point $c \in C$ over $s \in S$ one of the following cases holds:
    \begin{enumerate}
        \item The point $c$ is smooth and $\overline{M}_{C,c} \cong \overline{M}_{S,s}$. One can think about these points as ordinary smooth points.
        \item The point $c$ is smooth and $\overline{M}_{C,c} \cong \N \oplus \overline{M}_{S,s}$. These points can be considered as marked smooth points.
        \item The point $c$ is a node and $\overline{M}_{C,c} \cong \overline{M}_{S,s} \oplus \N e_1 \oplus \N e_2/(e_1+e_2=\delta_c)$ where $e_1, e_2$ map to regular parameters of both the branches of the nodes and $\delta \in \overline{M}_{S,s}$, which is called the smoothing parameter of the node.
    \end{enumerate}
\end{definition}

\begin{definition}
    The tropicalization of a logarithmic curve over an algebraically closed field $k$ with logarithmic structure $M_k$ consists of the dual graph $\Gamma_C$ in the sense of \citep[Section 1.1.]{morrison2007mori} of the underlying curve $C$ where each edge has the additional data of its smoothing parameter in $M_{k}$ that can be considered as the length of this edge. More precisely, $\Gamma_C$ has vertices given by the irreducible components of $C$, edges correspond to nodes between irreducible components and marked points at an irreducible components correspond to additional half-edges on the corresponding vertex.
\end{definition}

The definition of a logarithmic structure makes sense for any ringed topos, so in particular for algebraic stacks. Therefore, we can also make sense of the following theorem which can be found in \cite[Section 4]{kato2000log}:
\begin{theorem}
    The Moduli space $\mathcal{M}_{g,n}^{log}$ of log curves with $n$ markings and genus $g$ is represented by the Moduli stack of nodal curve $\Mgnbar$ with logarithmic structure associated to the boundary divisor coming from the locus of nodal curves.
\end{theorem}

We want to understand better how $\mathcal{M}_{g,n}^{log}$ looks like locally. The answer lies in nuclear log curves in the sense of \citep[Section 3.10.]{holmes2023models}. In particular, we will see that locally $\mathcal{M}_{g,n}^{log}$ consists of log curves that are nuclear. Instead of giving the precise definition, we will state the main properties of nuclear log curves as they are summarized in \citep[Section 2]{chen2025tale}.

\begin{theorem}
\begin{enumerate}
    \item For every log curve $C \to S$ there exists a strict étale cover $\cup_{i \in I} S_i \to S$ of the base such that for each $i \in I$ the base change $C \times_S S_i \to S_i$ is nuclear.
    \item For a nuclear log curve $C \to S$, the logarithmic stratification of $S$ has a unique closed stratum and for any $s \in S$ in the closed stratum the restriction constitutes an isomorphism $\Gamma(S,\overline{M}_S) \to \overline{M}_{S,s}$.
    \item We consider $s, s' \to S$ geometric points of a nuclear log curve $C \to S$ with $s$ in the closed stratum. The dual graph $\Gamma_{s'}$ of $C_{s'}$ is obtained from the dual graph $\Gamma_s$ of $C_s$ by mapping all smoothing parameters by the map $\overline{M}_{S, s} \to \overline{M}_{S, s'}$ and contracting those that are mapped to $0$.
\end{enumerate}
\end{theorem}

\begin{remark}
    We have seen that for a nuclear log curve $C \to S$ all the tropical information is in the closed stratum, we write $\overline{M}_S \coloneqq \overline{M}_{S,s}$ for $s$ in the closed stratum and call $\Gamma_s$ the dual graph $\Gamma_C$ of $C$ with edges labeled by elements of $\overline{M}_S$.
\end{remark}

Next, we want to extend the Abel-Jacobi section. This comes from \citep[Section 4.3.]{marcus2020logarithmic}. For a log curve $\pi_C: C \to S$ we can form the following commutative diagram with exact rows:
 \begin{center}
    \begin{tikzcd}
       M^{gp}_S \arrow[r] \arrow[d] &  \overline{M}^{gp}_S  \arrow[r] \arrow[d] & 0 \arrow[d]\\
       (\pi_C)_* M^{gp}_C \arrow[r] & (\pi_C)_* \overline{M}^{gp}_C  \arrow[r] & R^1 (\pi_C)_* \mathcal{O}^\times_C
    \end{tikzcd}
 \end{center}

Taking quotients of each column yields an exact sequence
 \begin{center}
    \begin{tikzcd}
         (\pi_C)_* M^{gp}_C/ M^{gp}_S \arrow[r]  &  (\pi_C)_* \overline{M}^{gp}_C/  \overline{M}^{gp}_S \arrow[r] &  R^1 (\pi_C)_* \mathcal{O}^\times_C
    \end{tikzcd}
 \end{center}

So a global section $\ell_S$ of $ (\pi_C)_* \overline{M}^{gp}_C/  \overline{M}^{gp}_S$ is mapped to a line bundle $\mathcal{O}_C(\ell_S)$ on $C$, defined up to tensor product with line bundles pulled back from $S$. Note that when $C$ is smooth, then global sections of $ (\pi_C)_* \overline{M}^{gp}_C/  \overline{M}^{gp}_S$ are divisors supported on the markings, so they correspond to assigning integers to each marked point. In this case, the procedure corresponds to the standard Abel-Jacobi section mapping the divisor $D=\sum_{i=1}^n m_i p_i$ to $\mathcal{O}_C(D)$.

In order to extend the Abel-Jacobi section, one might now be tempted to construct a Moduli space parametrizing log curves together with global sections of $ (\pi_C)_* \overline{M}^{gp}_C/  \overline{M}^{gp}_S$. This is exactly what Marcus and Wise did when they constructed a Moduli space of tropical divisors $Div_{g,a}$ in \citep[Section 4]{marcus2020logarithmic}. However, for the purpose of this paper, we will need a modification of this Moduli space. To understand this modification, we want to understand better how we can think of global sections of $ (\pi_C)_* \overline{M}^{gp}_C/  \overline{M}^{gp}_S$.

For now, let $\pi: C \to S$ be a nuclear log curve. As in \citep[Remark 7.3.]{cavalieri2020moduli} one can consider global sections $\ell_S$ of $\overline{M}^{gp}_C$ to be piecewise linear functions on the dual graph of $\Gamma_C$ of $C$: To each irreducible component of $C$, so to each vertex of $\Gamma_C$, the section $\ell_S$ assigns an element of $\overline{M}^{gp}_{S}$ such that whenever there is an edge $e$ in $\Gamma_C$ between $v_1$ and $v_2$, then $\ell_S(v_1)-\ell_S(v_2)= \ell_e \delta_e$ for some integer $\ell_e$ and $\delta_e$ the smoothing parameter of that edge. Furthermore, to each marking an integer is assigned. So sections of $(\pi_C)_* \overline{M}^{gp}_C/  \overline{M}^{gp}_S$ correspond to piecewise linear functions on the dual graph up to simultaneously adding a global constant to all vertices. 

Following \citep[Section 5]{marcus2020logarithmic}, we call such a section \emph{aligned} when at each geometric point $s \to S$ for each pair $(v,w)$ of vertices in $\Gamma_{C_s}$ either $\ell_S(v)-\ell_S(w)$ or $\ell_S(w)-\ell_S(v)$ are in $\overline{M}_{S}$, so the set of values of $\ell_S$ is totally ordered. This motivates the following definition that is equivalent to that given in \citep[Section 5]{marcus2020logarithmic}:
\begin{definition}
 For $\Xi=(\xi_1, \dots, \xi_b)$ an integer vector of length $b$, the Moduli functor $\RubhbX$ from the category of logarithmic schemes has as $S$-points pairs $(\pi_C: C \to S, \ell_S)$, where $\pi_C: C \to S$ is a log curve in $\mathcal{M}^{log}_{h,b}$ and $\ell_S \in \Gamma ((\pi_C)_* \overline{M}^{gp}_C/  \overline{M}^{gp}_S)$ is an aligned section that assigns $\xi_i$ to each marked point.
\end{definition}

\begin{remark}
    We always can add elements of $\overline{M}^{gp}_S$ to $\ell_S$, so the value of vertices on the maximal level becomes $0$. After this choice, $\mathcal{O}_C(\ell_S)$ can be considered as a line bundle $C$ and not just defined up to tensor product with line bundles pulled back from the base. For more details on this alternative viewpoint of $\RubhbX$ see \citep[Section 2]{chen2025tale}. This motivates to call the other levels which are below the maximal level the negative levels.
\end{remark}

Marcus and Wise further prove the representability of $\RubhbX$:
\begin{theorem}
    The functor $\RubhbX$ is representable by a proper Deligne-Mumford stack with a logarithmic structure. It is log étale over $\mathcal{M}_{h,b}^{log}$.
\end{theorem}

The stack $\RubhbX$ can be considered as the Moduli space of slopes of piecewise linear functions on log curves. We want to make this statement more precise:

\begin{definition}
    Let $\pi: C \to S$ be a log curve. A PL function on $C$ is a global section of $\overline{M}_C^{gp}$. A combinatorial PL function consists of:
    \begin{enumerate}
        \item a function $\ell': V(\Gamma) \to \overline{M}^{gp}_{S,s}$
        \item a function $\kappa: H(\Gamma) \to \Z$
    \end{enumerate}
    where $V(\Gamma)$ is the set of vertices and $H(\Gamma)$ the set of half edges of $\Gamma$ such that for two half edges $h_1,h_2$ forming an edge $e$ between vertices $v_1$ and $v_2$ we have
    \begin{align*}
        \kappa(h_2) \delta_e= \ell'(v_2)-\ell'(v_1)
    \end{align*}
    for $\delta_e$ the smoothing parameter of $e$.
\end{definition}

\begin{lemma}
    For a nuclear log curve, assigning to each vertex of $\Gamma$ the element from $\overline{M}^{gp}_{S,s}$ coming from a global section of $\overline{M}_C^{gp}$ gives a bijection between PL functions and combinatorial PL functions. In particular, an object of $\RubhbX$ defines slopes $\kappa$ on the edges of $\Gamma$.
\end{lemma}
\begin{proof}
    This follows from \citep[Lemma 2.12.]{chen2025tale}.
\end{proof}

For a nuclear base, we can construct even more data out of a section of $\RubhbX$ as described in \citep[Section 5]{chen2025tale}:
\begin{definition} \label{def: enhanced_graph}
    Let $(\pi: C \to S, \ell_S)$ be an object of $\RubhbX (S)$ with $S$ nuclear. The enhanced level graph of $C$ is the dual graph of $C$, so the dual graph of a curve in its closed stratum, together with the partial ordering of the vertices induced by $\ell_S$. The set of levels will be denoted by $\{0, \dots, -N\}$, so we can consider the enhancement to be a function $V(\Gamma) \to \{0, \dots, -N\}$.
\end{definition}

A simple rescaling ensemble assigns to each level $i \in  \{0, \dots, -N\}$ of the dual graph a section of the monoid $s_i \in H^0(S, \mathcal{O}_S)$ satisfying some compatibility restrictions described in \citep[Section 5.2.]{chen2025tale}. There, one can also find an explicit description about how by using a so-called log-splitting, out of the section $\ell_S$, one can construct a simple rescaling ensemble.

\begin{lemma} \label{lem: rescaling_ensamble} 
For each edge in $\Gamma$ between vertices of level $j$ and $i$ with $j < i$ and smoothing parameter $\delta_e$ and slope $\kappa_e$, the rescaling ensemble satisfies:
    \begin{align*}
        \alpha(\delta_e)^{\kappa_{e}}= s_j \dots s_{i-1} \in \mathcal{O}_S / \mathcal{O}_S^{\times} 
    \end{align*}
\end{lemma}
\begin{proof}
This is \citep[Lemma 5.6.]{chen2025tale}.
\end{proof}
The minimal logarithmic structures in the sense of \cite{gillam2012logarithmic} are described in \citep[Section 4]{chen2025tale}. They use a slightly different version of $\RubhbX$ but the arguments they provide also work for our Moduli stack:
\begin{theorem}
    The minimal logarithmic structure on the base for an object $(\pi_C: C \to S, \ell_S)$ of $\RubhbX$ is freely generated by symbols corresponding to negative levels and edges of the dual graph of $C$ of slope $0$.
\end{theorem}

\subsection{Hurwitz Spaces in Characteristic \texorpdfstring{$0$}{0}}
In this section, we want to recall known results about logarithmic Hurwitz spaces in characteristic $0$ as they are described in \citep[Section 3]{mochizuki1995geometry}. Let $K$ be a field of characteristic $0$ with trivial logarithmic structure, so $M_K =K^\times$.
\begin{definition}
    We say that a map between marked nodal curves 
    \begin{align*}
    f:(C, p_1, \dots, p_n) \to (D, q_1, \dots ,q_n)    
    \end{align*}
    that maps markings to markings does preserve the order if from $i<j$ it follows that for $q_{i'}=f(p_i)$ and $q_{j'}=f(p_j)$ it follows that $i'< j'$.
\end{definition}
\begin{definition}
    A logarithmic Hurwitz cover of type $A=(h,g,N,d, \Lambda)$, where $h,g,N,d$ are integers and $\Lambda=(\lambda_1, \dots, \lambda_b)$ is a vector of integers of length $b$, over a logarithmic scheme $S$ over $K$ is a morphism of logarithmic schemes $f: C \to D$ over $S$ such that
    \begin{enumerate}
        \item $C$ is a logarithmic curve of genus $h$ with $b$ markings over $S$.
        \item $D$ is a logarithmic curve of genus $g$ with $N$ markings over $S$.
        \item $f: C \to D$ is log étale, such that the underlying morphism is  finite, surjective and of degree $p$ mapping markings to markings preserving the order and such that the map at the $i$-th marking is $\lambda_i$ to one. Furthermore, all preimages of a marked point are marked.
    \end{enumerate}
The Moduli functor assigning to a logarithmic scheme the set of logarithmic Hurwitz covers of type $A=(h,g,N,p, \Lambda)$ over it is denoted by $\mathcal{LA}dm_A$.
\end{definition}
\begin{remark}
    This definition implies that the underlying morphism of a logarithmic Hurwitz cover is an admissible cover in the sense of \citep[Section 4]{abramovich2003twisted}. From this, we also see that $\mathcal{LA}dm_A$ is nonempty iff the Riemann-Hurwitz formula 
    \begin{align*}
        2h-2=d(2g-2)+\sum_{i=1}^b (\lambda_i-1)
    \end{align*}
    holds.
\end{remark}
The next theorem is from \citep[Section 3.22.]{mochizuki1995geometry}:
\begin{theorem}
    The Moduli functor $\mathcal{LA}dm_A$ is representable by the proper Deligne-Mumford stack of admissible covers from \citep[Section 4]{abramovich2003twisted} with a logarithmic structure.
\end{theorem}
Mochizuki further describes the minimal logarithmic structure on the base $S$ for a logarithmic admissible cover in the sense of \cite{gillam2012logarithmic} to be the following: Let $M^C_S$ be the minimal logarithmic structure on the base for the source curve $C$ and $M^D_S$ be the minimal logarithmic structure on the base for the target curve $D$. For a logarithmic Hurwitz cover $f: C \to D$, the minimal logarithmic structure $M_S^f$ on the base is $M^C_S \oplus_{\mathcal{O}_S^\times} M^D_S/ \sim$. The relation $\sim$ identifies for each pair $(r,q)$ where $r$ is a marked point or node of $C$ with corresponding parameter $e_r \in M^C_S$ that maps to a marked point or node $q$ of $D$ with corresponding parameter $e_q \in M^D_S$ by multiplicity $n_r$, the smoothing parameter $e_q $ with $n_r e_r$.  

Harris and Mumford describe in \citep[p.62]{harris1982kodaira} an étale neighborhood of a logarithmic Hurwitz cover $f: C \to D$ in the underlying stack of $\mathcal{LA}dm_A$. From their description it follows that these étale neighborhoods are locally isomorphic to $Spec(K[\overline{M}_S^f])$, so in particular $\mathcal{LA}dm_A$ is locally isomorphic to a toric variety. The next theorem now follows by \citep[Theorem 4.7.]{kato1996log}, stating that toric varieties are logarithmically smooth over a ground field with trivial logarithmic structure. It is also stated in \citep[Section 3.22.]{mochizuki1995geometry}:
\begin{theorem} \label{thm: gen_fibre_smooth}
    The Moduli stack $\mathcal{LA}dm^d_A$ is of dimension $3g-3+N$ and log smooth over $K$ with trivial logarithmic structure.
\end{theorem}
\subsection{Exact and Quasi-Exact Differential Forms}
In this subsection, we study some classes of differential forms on curves that are motivated from Berkovich analytic geometry and are introduced in \citep[3.4.7]{brezner2017lifting}. They will be relevant when discussing covers in positive characteristic. Note that we call \emph{quasi-exact} those forms that in the work of Brezner and Temkin were called \emph{mixed}. While Brezner and Temkin only introduced this notion for the geometric Frobenius, we will need it in more generality:
\begin{definition} \label{def:exact_quexact}
    Let $f: C \to D$ be a finite surjective morphism of degree $p$ between log smooth curves over a $\Z_{(p)}$-log scheme $S$. Furthermore, let $q$ be a point of characteristic $p$ of $C$ which is in the generic locus and let $f$ be inseparable at $q$. We consider, $\psi$ which is a meromorphic section, in the sense of \cite[\href{https://stacks.math.columbia.edu/tag/01X1}{Tag 01X1}]{stacks-project}, of the relative dualizing sheaf $\omega^{log}_f$ that is defined to be
    \begin{align*}
        \omega^{log}_f \coloneqq (f^*\Omega_D^{log})^\vee \otimes_{\mathcal{O}_C} \Omega^{log}_C.
    \end{align*} 
    
    Such sections will be called \emph{bivariant forms}. Étale locally around $q$ the morphism $f$ is given by a ring map $R[x] \to T$ over $R$, with $T$ étale over $R[y]$ for some $y$ such that $f$ is given by $x \mapsto y^p+g$ with $g$ vanishing at $f(q)$. On these neighborhoods, the form $\psi$ can be written as $\psi= (d_Dx)^\vee \otimes \phi$ for some meromorphic section $\phi$ of $\Omega_C$.
    \begin{enumerate}
        \item $\psi$ is called \emph{exact} at $q$ if one can achieve $\phi=d_Ch$ for some meromorphic function $h$ on $C$.
        \item $\psi$ is called \emph{quasi-exact} at $q$ if one can achieve $\phi=u (y^{p-1}d_Cy+d_Ch)$ for some meromorphic function $h$ on $C$ and $u$ a unit in $R$.
    \end{enumerate}
    The differential form $\psi$ will be called \emph{exact} or \emph{quasi-exact} respectively at a component of $C$, if it is exact or quasi-exact at all smooth points of that component.
\end{definition}
\begin{remark}
    A priori, it is not clear that this definition is independent of choices. This will follow later from our discussion of the Cartier operator in Theorem~\ref{thm: rel_car_prop}.
\end{remark}

We will need this formulation for the definition of our stack and will need to analyze it further for the case of $f$ being the relative Frobenius map. For this, recall the following diagram linking different Frobenii:
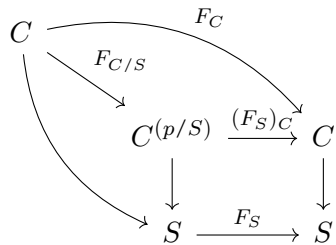
\begin{center}
    \begin{tikzcd}
C
\arrow[drr, bend left, "F_{C}"]
\arrow[ddr, bend right]
\arrow[dr, "F_{C/S}"] & & \\
& C^{(p/S)} \arrow[r, "(F_S)_C"] \arrow[d ]
& C \arrow[d] \\
& S \arrow[r, "F_S"]
& S
\end{tikzcd}
\end{center}

In the following, we will consider the case of smooth curves over an algebraically closed field $k$, so we can identify the Frobenius twist of a curve with the curve itself. An important tool to study exactness is the Cartier operator that was first introduced in \cite{cartier1957nouvelle}:
\begin{definition}
   Let $C/k$ be a smooth curve and $\omega$ be a meromorphic differential form on an open subset $U$ of $C$. Furthermore, let $t$ be a local parameter for a point $x$ on $C$, so a uniformizer of the local ring $\mathcal{O}_{C,x}$. Then étale locally around $x$ one can write
   \begin{align*}
      \omega = \sum_{i=0}^{p-1} f_i^p t^i dt. 
   \end{align*}
   for some rational functions $f_i$. The Cartier operator $\mathfrak{c}$ maps $\omega$ to the differential form locally given by $f_{p-1}dt$.
\end{definition}
 The fact that our local description glues to a morphism of differential forms is a classical result that can be found in \cite{cartier1957nouvelle}.

    When working in families, one can define the Cartier operator similarly by
    \begin{align*}
        \sum_i a_i t^i dt \mapsto \sum_{j}{a_{(j-1)p -1}} t^j dt
    \end{align*}
    The differential form we constructed will not be a differential form on $C$, but a differential form on the Frobenius twist $C^{(p/S)}$, in genus $0$ however, they are isomorphic. 
\begin{definition}\label{def: cartier_family}
    The assignment defined above will be called the relative Cartier operator. It coincides with the Cartier operator coming from an algebraically closed field after composing with the isomorphism $C^{(p/k)} \cong C$. This allows us to denote the relative Cartier operator by $\mathfrak{c}$ as well.
\end{definition}
    It is a classical result that $\mathfrak{c}(\omega)$ is well defined, so in particular is independent of the choice of local coordinates as we will see in the next theorem.

\begin{theorem} \label{thm: prop_cartier}
Let $\omega, \omega'$ be rational differential forms on $U \subseteq C$ and $f$ be a meromorphic function on $U$. Then the Cartier Operator glues to a morphism between sheaves of meromorphic differentials on $C$:
\begin{enumerate}
    \item $\omega$ is the differential of a meromorphic function iff $\mathfrak{c}(\omega)=0$.
    \item $\mathfrak{c}(\omega + \omega') = \mathfrak{c}(\omega)+\mathfrak{c}(\omega')$.
    \item $\mathfrak{c}(f^p \omega)= f \mathfrak{c}(\omega)$.
    \item $\mathfrak{c}$ maps $\Omega_{C/S}(\sum_{i=1}^n m_i p_i)$ surjectivly to $\Omega_{C^{(p/S)}/S}(\sum_{i=1}^n \left \lceil{\frac{m_i}{p}}\right \rceil p_i')$
     where $p_i'$ is the image of the $i$-th marking $p_i$ under the relative Frobenius map.
\end{enumerate}
\end{theorem}
\begin{proof}
These are well known results that can be found in the classical literature for the case $S=k$, for example in \cite{cartier1957nouvelle}, \cite{cartier1958questions} or \cite{seshadri1958operation}. The case for $S$ being a general base was introduced in \citep[Section 7]{katz1970nilpotent} and more extensively explained in \citep[Section 4.2.]{achinger2021global}.
\end{proof}
\begin{remark} \label{rem: cartier_linear}
    The findings from the last lemma tell us that we can consider the Cartier operator to be a surjective map of vector bundles
    \begin{align*}
        \mathfrak{c} \!: (F_{C/S})_*\Omega_{C/S}(\sum_{i=1}^n m_i p_i) \to \Omega_{C^{(p/S)}/S}\left(\sum_{i=1}^n \left \lceil{\frac{m_i}{p}}\right \rceil p'_i\right)
    \end{align*}
   
\end{remark}
We will need the Cartier operator in this generality only to see that the locus of zeroes and poles of exact and quasi-exact forms define a closed substack of $\mathcal{M}_{g,n}$.

Using the Cartier operator, we can now prove that when working over an algebraically closed ground field $k$ and $f$ being the relative Frobenius map, then our notion of exactness and quasi-exactness from \ref{def:exact_quexact} agrees with the notion of exactness and quasi-exactness coming from \citep[3.4.7.]{brezner2017lifting}. We do an explicit computation for the case of $\mathbb{A}^1$ and the general case follows from Theorem~\ref{thm: prop_cartier} as it tells us that we can check exactness étale locally.
\begin{lemma}
    When working with the affine line over an algebraically closed field $k$, we choose coordinates $x$ and $y$ such that the relative Frobenius map $C=\mathbb{A}^1_k \to \mathbb{A}^1_k=D$ comes from the ring map
    \begin{align*}
        k[x] &\to k[y] \\
        x &\mapsto y^p
    \end{align*}
    Then a bivariant differential form $\psi$ is exact iff it can be written as
    \begin{align*}
        (d_D g)^\vee \otimes d_C h \
    \end{align*}
    for $g \in k(x)$ and $h \in k(y)$. 
    
    The form is quasi-exact, iff it can be written as
    \begin{align*}
        (d_D w^p)^\vee \otimes u(d_C h + w^{p-1}d_Cw)
    \end{align*}
    for $h \in k(y), u \in k^\times$, and $w \in k(y)$ that is not a $p$-th power.
\end{lemma}
\begin{proof}
    By choosing $g=x$ in the exact case and $w=y$ in the quasi-exact case, we see that exact and quasi-exact forms from Definition~\ref{def:exact_quexact} can be written as stated.
    
   We are left to show that forms that can be written as in the statement are exact or quasi-exact respectively in the sense of Definition~\ref{def:exact_quexact}. One can write
    \begin{align*}
        (d_{D} g)^\vee \otimes (d_{C} h) = (d_D x)^\vee \otimes \frac{d_C h}{\frac{\partial}{\partial x}g}
    \end{align*}
    But $\frac{\partial}{\partial x}g$ as a function in $y$ is a $p$-th power, so we have
    \begin{align*}
        \mathfrak{c}\left(\frac{d_C h}{\frac{\partial}{\partial x}g}\right) = \frac{\mathfrak{c}(d_C h)}{\frac{\partial}{\partial x}g}=0
    \end{align*}
    so $\frac{d_C h}{\frac{\partial}{\partial x}g}$ is the derivative of a function as desired. 

    For the case of quasi-exact forms it holds
    \begin{align*}
        (d_D w^p)^{\vee} \otimes (d_{C} h+ w^{p-1}d_
        {C}w)= (d_{D} y^p)^\vee \otimes \frac{d_Ch+ (\frac{\partial}{\partial y} w)w^{p-1}(d_Cy)}{\frac{\partial}{\partial x} w^{p}}
    \end{align*}
    and the statement follows once we have
    \begin{align*}
        \mathfrak{c} \left(\frac{(\frac{\partial}{\partial y} w)w^{p-1}(d_Cy))}{\frac{\partial}{\partial x} w^{p}}\right) = d_{C} y
    \end{align*}
    For this note 
    \begin{align*}
        \mathfrak{c} \left(\frac{(\frac{\partial}{\partial y} w)w^{p-1}(d_Cy))}{\frac{\partial}{\partial x} w^{p}}\right) =
        \frac{w}{(\frac{\partial}{\partial x} w^{p})^{\frac{1}{p}}} \mathfrak{c} \left(\frac{(\frac{\partial}{\partial y} w)(d_Cy))}{w}\right) =  \frac{(\frac{\partial}{\partial y} w)}{(\frac{\partial}{\partial x} w^{p})^{\frac{1}{p}}} d_{C} y
    \end{align*}
    and after writing $w$ in partial fraction decomposition we have
    \begin{align*}
        w &= \sum_i \sum_j a_{i,j} (y-c_j)^i \\
        w^p &= \sum_i \sum_j a_{i,j}^p (x-c^p_j)^i \\
        \frac{\partial}{\partial y} w &= \sum_i \sum_j ia_{i,j} (y-c_j)^{i-1} \\
        \frac{\partial}{\partial x} w^p&= \sum_i \sum_j ia_{i,j}^p(x-c_j^p)^{i-1}
    \end{align*}
    So as the Frobenius morphism fixes $\mathbb{F}_p$, we have 
    \begin{align*}
        (\frac{\partial}{\partial y} w)^p= \left(\sum_i \sum_j ia_{i,j} (y-c_j)^{i-1} \right)^p =\sum_i \sum_j ia_{i,j}^p(x-c_j^p)^{i-1}=\frac{\partial}{\partial x} w^p
    \end{align*}
    and the statement follows.
\end{proof}

As we will work with bivariant differential forms, we will twist the Cartier operator so we can analyze them by the same techniques.

\begin{definition} \label{def:twisted_cartier}
    Twisting the Cartier operator by $(\Omega_{C^{p/S}/S})^\vee$ gives a map of vector bundles on $C$
    \begin{align*}
        \mathfrak{tc} \!: (F_{C/S})_* \Omega_{F_{C/S}}(\sum_{i=1}^n m_i p_i) \to \mathcal{O}_{C^{(p/S)}} \left(\sum_{i=1}^n \left \lceil{\frac{m_i}{p}}\right \rceil p'_i\right).
    \end{align*}
\end{definition}

For concrete calculations we want to perform in following chapters, working with the projective line over the base $k$ is enough, so for the rest of this section we will focus on the case $C=\Pjk$.
\begin{lemma} \label{lem: car_sur}
    On global sections, the Cartier operator for the projective line
       \begin{align*}
        H^0(\mathfrak{c}) \!: H^0\left((F_{\Pjk/k})_*\Omega_{\Pjk}(\sum_{i=1}^n m_i p_i)\right) \to H^0\left(\Omega_{\Pjk}\left(\sum_{i=1}^n \left \lceil{\frac{m_i}{p}}\right \rceil p_i\right)\right)
    \end{align*}
    is surjective.
\end{lemma}
\begin{proof}
    We know from \cite[Proposition 9]{serre1958topologie} that the Cartier Operator corresponds to the absolute Frobenius map under Serre duality, so we can identify the cokernel on global sections of the Cartier operator with the kernel of the map
    \begin{align*}
        H^1((F_{\Pjk/k})_*) \!: H^1\left(\mathcal{O}_{\Pjk}\left(\sum_{i=1}^n\left \lfloor \frac{-m_i}{p}\right \rfloor[p_i]\right)\right) \to H^1\left((F_{\Pjk/k})_*(\mathcal{O}_{\Pjk})(\sum_{i=1}^n-m_i[p_i])\right).
    \end{align*}
    Using the convention $(y-\infty)=y^{-1}$ we define
    \begin{align*}
        d&=\sum_{i=1}^n \left \lfloor \frac{-m_i}{p}\right \rfloor \\
        d'&=-\sum_{i=1}^n m_i
    \end{align*}
    and
    \begin{align*}
        f&=\frac{\prod_{i=1}^n(y-p_i)^{\left \lceil \frac{m_i}{p}\right \rceil}}{y^d} \\
        g&=\frac{\prod_{i=1}^n (y-p_i)^{m_i}}{y^{d'}}
    \end{align*}
    we get a commutative diagram
     \begin{center}
    \begin{tikzcd}
       \mathcal{O}_{\Pjk}(\sum_{i=1}^n \left \lfloor \frac{-m_i}{p} \right \rfloor [p_i]) \arrow[r, "F_{\Pjk /k}"] \arrow[d, "f", "\sim"'] &  (F_{\Pjk/k})_*(\mathcal{O}_{\Pjk})(p\sum_{i=1}^n  \left \lfloor \frac{-m_i}{p} \right \rfloor [p_i])   \arrow[hookrightarrow, r] \arrow[d, "f^p", "\sim"'] & ({F_{\Pjk/k}})_*(\mathcal{O}_{\Pjk})(-\sum_{i=1}^n m_i[p_i])  \arrow[d, "g", "\sim"']\\
       \mathcal{O}_{\Pjk}(d) \arrow[r, "F_{\Pjk/k}"] & ({F_{\Pjk/k}})_*(\mathcal{O}_{\Pjk})(pd)  \arrow[hookrightarrow, r] & (F_{\Pjk /k})_*(\mathcal{O}_{\Pjk})(d')
    \end{tikzcd}
 \end{center}
 So it is enough to show that the lower row is injective on first cohomology. But here it is explicitly given by
\begin{center}
    \begin{tikzcd}
       (\frac{1}{T_0 T_1}k[T_0,T_1])_d \arrow[r, "F_{\Pjk /k}"] & (\frac{1}{T_0 T_1} (F_{\Pjk /k})_*(k)[T_0,T_1])_{pd}  \arrow[r, "T_1^{d'-pd}"]  & (\frac{1}{T_0 T_1}(F_{\Pjk} / k)_*(k)[T_0,T_1])_{d'}
       \end{tikzcd}
 \end{center}
 where $[T_0:T_1]$ are the standard homogeneous coordinates on $\Pjk$, which is clearly injective.

 Alternatively, one can also directly construct $f^p x^{p-1}dx$ as a preimage of $f dx$ under the Cartier operator and compare the order of zeroes and poles.
\end{proof}

Next, we want to develop a better understanding of bivariant differential forms: On the projective line over an algebraically closed field we can describe the twisted Cartier operator more explicitly:

\begin{lemma}
    Let $F_{\Pjk/k}: \Pjk \to \Pjk$ be the relative Frobenius and $x,y$ be rational functions on the target such that $F_{\Pjk/k}$ corresponds to the field extension $k(x) \to k(y)$ with $x \mapsto y^p$. The twisted Cartier operator can be expressed as
    \begin{align*}
        \mathfrak{tc}: (F_{\Pjk/k})_* \Omega_{F_{ \Pjk/k}}(\sum m_i p_i) &\to \mathcal{O}_{\Pjk}\left(\sum_{i=1}^n \left \lceil{\frac{m_i}{p}}\right \rceil p_i\right) \\
        \sum_{i=0}^{p-1} f_i^py^i (d_{{\Pjk}^{(p/k)}} x)^\vee \otimes (d_{\Pjk} y)&\mapsto f_{p-1}
    \end{align*}
    Whenever there is no danger of misinterpretation, we will write $\frac{dy}{dx}$ instead of $(d_{{\Pjk}^{(p/k)}} x)^\vee \otimes (d_{\Pjk} y)$.
\end{lemma}

The calculations we have done hold étale locally for curves of higher genus as well. Therefore, one sees that the properties of the Cartier operator from Theorem~\ref{thm: prop_cartier} and Lemma~\ref{lem: car_sur} extend to the twisted Cartier operator:
\begin{theorem} \label{thm: rel_car_prop}
Let $\omega_1, \omega_2 \in \Omega_{F_{C / S}}$ be meromorphic differential forms and $f \in \mathcal{O}_{C}$. 
    \begin{enumerate}
        \item $\rc$ is well defined.
        \item $\rc(\omega_1+\omega_2)=\rc(\omega_1)+\rc(\omega_2)$.
        \item $\rc(f^p \omega_1)=f\rc(\omega_1)$.
        \item $\omega_1$ is exact iff $\rc(\omega_1)=0$.
        \item $\omega_1$ is quasi-exact iff $\rc(\omega_1)$ is a unit coming from $S$.
        \item $\rc$ maps $\Omega_{F_{C /S}}(\sum_{i=1}^n m_i p_i)$ surjectivly to $\mathcal{O}_{C^{(p/S)}}(\sum_{i=1}^n \left \lceil{\frac{m_i}{p}}\right \rceil p'_i)$.
        \item For the projective line over $k$ the twisted Cartier operator on global sections
        \begin{align*}
        H^0(\mathfrak{tc}): H^0(((F_{\Pjk /k})_*\Omega_{F, \Pjk})(\sum_{i=1}^n m_i p_i)) \to H^0(\mathcal{O}_{\Pjk}\left(\sum_{i=1}^n \left \lceil{\frac{m_i}{p}}\right \rceil p_i\right))
    \end{align*}
    is surjective.
    \end{enumerate}
\end{theorem}

\section{The Moduli Spaces \texorpdfstring{ $\LHURphgnZ$}{LHZA} and \texorpdfstring{ $\LHURphgnF$}{LHFAXi}} \label{sec:mod_space}
In this section, we want to define the Moduli spaces of liftable logarithmic Hurwitz covers both in mixed and in equicharacteristic. This is motivated by the main result of \cite{brezner2017lifting} where certain covers of degree $p$ over an algebraically closed field $k$ of characteristic $p$ were identified to be exactly those that can be lifted to a cover of smooth curves. We give $\Z_{(p)}$ the standard logarithmic structure on a discrete valuation ring which means we consider the logarithmic structure that maps the generator of the monoid $\N$ to the unique maximal ideal $(p)$ by mapping $1$ to $p$ and consider $\mathbb{F}_p$ to carry the trivial logarithmic structure.
\subsection{Definition and Representability}
Using Gillam's criteria from \cite{gillam2012logarithmic}, we will at first construct an auxiliary Moduli stack as a category fibered over the category of schemes and then describe minimal objects that allow us to consider the stack as a category fibered over the category of  logarithmic schemes. 

As in the case of Hurwitz covers in characteristic $0$, we will fix integers $h,g,b,N$, an integer vector $\Lambda=(\lambda_1, \dots \lambda_b)$ determining the multiplicity of the map at each marking. We denote $A=(h,g,b,N,\Lambda)$ to store this data. In the equicharacteristic case, we consider an additional integer vector $\Xi=(\xi_1, \dots, \xi_b)$ measuring how much the ramification indices of the cover differ from the multiplicities. The $\xi_i$ can be thought as the logarithmic order of zeroes and poles of sections of the relative dualizing sheaf.

Non-trivial choices for $\Xi$ can lead to interesting Moduli spaces supported only over the special fiber, but as in characteristic $0$ the ramification indices are determined by the multiplicities, whenever one wants to work in mixed characteristic, one has to assume that all $\xi_i$ are $0$. It is also worth noting that as in the case of characteristic $0$, for the Moduli space we are about to construct to be non-empty, a Riemann-Hurwitz formula has to be fulfilled:
\begin{align*}
    2h-2=p(2g-2)+\sum_{i=1}^b(\lambda_i + \xi_i-1).
\end{align*}
At first, we have to construct some auxiliary Moduli spaces.

Let $B$ be either $\Z_{(p)}$ or $\mathbb{F}_p$ with log structure defined as above. The first case will be called the mixed characteristic case, while we will call the second one the equicharacteristic case. 
\begin{definition}
    The Moduli functor $\mathcal{L}\Maphgnb$ from the category of logarithmic schemes over $B$ to the category of sets has as $(S,M_S)$-points morphisms $f_S: (C, M_C) \to (D,M_D)$ such that:
    \begin{enumerate}
        \item $(C,M_C)$ is a $b$-marked log curve of genus $h$ over $(S,M_S)$
        \item $(D,M_D)$ is a $N$-marked log curve of of genus $g$ over $(S,M_S)$.
        \item $f: (C,M_C) \to (D,M_D)$ is a morphism of logarithmic schemes over $(S,M_S)$ that maps markings to markings and preserves their order.
    \end{enumerate}
    Morphisms are mapped to pullback diagrams.   
\end{definition}
Similar to \cite{gillam2012logarithmic}, we want to relate Moduli functors on the category of logarithmic schemes to Moduli functors on the category of schemes, for this we will need the following definition:
\begin{definition}
    Let $\mathcal{M}$ be a category fibered in groupoids over the category of logarithmic schemes. By $\mathbf{LOG}(\mathcal{M})$ we denote the category fibered in groupoids over the category of schemes which has as $S$-points pairs $(M_S, X)$, where $M_S$ is a logarithmic structure on $S$ and $X$ is an object of $\mathcal{M}$ over $(S,M_S)$.
\end{definition}
\begin{lemma}
    $\mathbf{LOG}(\mathcal{L}\Maphgnb)$ is representable by an algebraic stack.
    \end{lemma}
\begin{proof}
    Note that ${\mathcal{M}^{ps-log}_{g,N}}$, the Moduli space of $N$-marked prestable log curves of genus $g$, is representable by a logarithmic algebraic stack with universal family ${\mathcal{C}^{ps-log}_{g,N}}$. According to \citep[Corollary 1.1.2]{wise2016moduli} there exists an Artin stack $\mathcal{M}^{log}_{h,b}(\mathcal{C}^{ps-log}_{g,N}/{\mathcal{M}^{ps-log}_{g,N}})$ with a logarithmic structure that parametrizes logarithmic maps from stable $N$-marked log curves of genus $h$ into $N$-marked log curves of genus $g$. 
    
    The stack $\mathbf{LOG}(\mathcal{L}\Maphgnb)$ now is cut out from the stack $\mathbf{LOG}(\mathcal{M}^{log}_{h,N}(\mathcal{C}^{ps-log}_{g,N}/\mathcal{M}^{ps-log}_{g,N}))$ as a closed substack by the condition on how the markings map.
\end{proof}
%The underlying algebraic stack is described in \citep[Theorem 3.12]{rydh2011representability} to be locally of finite presentation.
\begin{definition}
    The subfunctor $\mathcal{LOG}\Mapfshgn \subseteq \mathbf{LOG}(\mathcal{L}\Maphgnb)$ has as $S$-points those maps $f_S: (C,M_C) \to (D,M_D)$ where the underlying morphism of schemes is finite of degree $p$, surjective, the preimage of the nodes of $D$ are the nodes of $C$, the multiplicity of both sides of a node agree, map the markings $\lambda_i$ to $1$ to their image and are at each point either log-étale or wild, i.e. at points of characteristic $p$ of degree $p$.
\end{definition}
\begin{lemma} \label{lem: logmap}
    $\mathcal{LOG}\Mapfshgn$ is representable by an algebraic stack.
\end{lemma}
\begin{proof}
    The simultaneous semistable reduction theorem for finite covers of curves that can for example be found in  \citep[Corollary 3.10.]{liu2006stable}, states that after a finite extension of the base, an object $f_\eta$ of $\mathcal{LOG}\Mapfshgn$ over the generic fiber of a discrete valuation ring $R$ can be uniquely extended to an object of $\mathcal{LOG}\Mapfshgn$ over $R$. From this, we can conclude that $\mathcal{LOG}\Mapfshgn$ is a closed substack of $\mathbf{LOG}(\mathcal{L}\Maphgnb)$.
\end{proof}

With these notions, we can finally construct the desired Moduli space which will be the main object of this paper.

\begin{definition}
    With $A$ and $\Xi$ as above, the functors of preliftable Hurwitz covers in mixed characteristic $\PLHURphgnZ$ or equicharacteristic $\PLHURphgnF$ are the logarithmic Moduli functors over $\mathbb{Z}_{(p)}$ with standard logarithmic structure and $\mathbb{F}_p$ with trivial logarithmic structure respectively, that have $S$-points:
    \begin{align*}
        (f_S: C_S \to D_S, \ell_S)
    \end{align*}
    \begin{enumerate}
        \item $C_S$ is a stable log curve over $S$ of genus $h$ with $b$ markings.
        \item $D_S$ is a semistable log curve over $S$ of genus $g$ with $N$ markings.
      
        \item $f_S$ is a finite surjective map of degree $p$ that maps markings to markings preserving the order with multiplicity given by $\Lambda$ and is at each point either logarithmic étale or wild, so it is of degree $p$ at points of characteristic $p$.

        \item $(C_S,\ell_S) \in \RubhbX(S)$ such that for the relative dualizing sheaf
        \begin{align*}
            \omega^{log}_{f_S} = (f_S^* \Omega^{log}_{D_S/S})^\vee \otimes_{\mathcal{O}_{C_S}} \Omega^{log}_{C_S/S}
        \end{align*}
        it holds $\omega^{log}_{f_S} (-\ell_S) \cong \mathcal{O}_{C_S}$. Note that in the mixed characteristic case, we assume $\Xi=(0, \dots, 0)$.
    \end{enumerate}
    Morphisms map to pullback diagrams.
\end{definition}
\begin{remark}
    Because of the logarithmic structure, the condition of the nodes mapping to nodes with the same multiplicity on both branches already follows from the given definition. This is described in \citep[Lemma 5.2.2.]{brezner2017lifting}.
\end{remark}

At first, we want to see that these Moduli functor have proper logarithmic Moduli stacks.

\begin{theorem} \label{thm:plhur_rep}
    The Moduli functors $\PLHURphgnZ$ and $\PLHURphgnF$ are representable by a logarithmic stack with their underlying stacks being proper Deligne-Mumford.
\end{theorem}
\begin{proof}

    By construction, $\mathbf{LOG}(\PLHURphgnZ)$ and $\mathbf{LOG} (\PLHURphgnF)$ are closed substacks of the product stack
    \begin{align*}
        \mathcal{LOG}\Mapfshgn \times_{\mathbf{LOG}({\mathcal{M}}^{log}_{h,b})} \mathbf{LOG} (\RubhbX)
    \end{align*}
    over their respective base.
    
      In particular, $\mathbf{LOG}(\PLHURphgnZ)$ and $\mathbf{LOG} (\PLHURphgnF)$ are representable by algebraic stacks and according to Gillam's criteria from \cite{gillam2012logarithmic} and \cite{wise2016moduli} the representability statement follows once we have shown the existence of coherent minimal logarithmic structures for objects and shown that the pullback of minimal objects is still minimal. The following construction of minimal logarithmic structures for objects of $\PLHURphgnZ$ and $\PLHURphgnF$ will follow similar arguments as provided in the construction of minimal logarithmic structures for the stack $\Div$ in \citep[Theorem 4.2.4.]{marcus2020logarithmic}.

    Let $(f:C \to D, \ell_S), (f':C' \to D', \ell'_{S'})$ and $(f'':C'' \to D'', \ell''_{S''})$ be objects of $\PLHURphgnZ$ or $\PLHURphgnF$ over $S, S'$ and $S''$ respectively such that the underlying schemes of $S, S'$ and $S''$ agree. The object $(f:C' \to D', \ell'_{S'})$ is called minimal if in any diagram of morphisms
 \begin{center}
    \begin{tikzcd}
       (f:C \to D, \ell_S) \arrow[r] \arrow[rd] & (f'':C'' \to D'', \ell''_{S''}) \arrow[d, dotted] \\
       & (f':C' \to D', \ell'_{S'})
    \end{tikzcd}
    \end{center}
    such that $S \to S'$ and $S \to S''$ is the identity on the underlying schemes there exists a dotted arrow that is also the identity on the underlying schemes. In other words, when $(f:C \to D, \ell_S)$ is pulled back both from $(f':C' \to D', \ell'_{S'})$ and $(f'':C'' \to D'', \ell''_{S''})$, then also $(f'':C'' \to D'', \ell''_{S''})$ is pulled back from $(f':C' \to D', \ell'_{S'})$ such that $(f:C \to D, \ell_S)$ is pulled back by the composition.

    As we have already discussed minimal logarithmic structures for the Hurwitz space in characteristic $0$ in Corollary~\ref{cor:min_log_stru_cov}, it is enough to study minimal logarithmic structures for nuclear neighborhoods of points $s$ of positive characteristic. Let $S$ be such a neighborhood.
    
    At first, note that there are source maps
    \begin{align*}
        \phi_{\mathbb{Z}_{(p)}}: \PLHURphgnZ \to \mathcal{RUB}^0_{h,b} \\
        \phi_{\mathbb{F}_p}: \PLHURphgnF \to \RubhbX \\   
    \end{align*}
    that just remember the pair $(C, \ell_S)$ and target maps
    \begin{align*}
        \delta_{\mathbb{Z}_{(p)}}: \PLHURphgnZ \to {\mathcal{M}^{ps-log}_{g,N}} \\
        \delta_{\mathbb{F}_p}: \PLHURphgnF \to {\mathcal{M}^{ps-log}_{g,N}}
    \end{align*}
    remembering only the target curve.
    Let $N_S$ be the logarithmic structure on $S$, $M_S^{Rub}$ be the minimal logarithmic structure on the base for $(C, \ell_S)$ as a section of $\RubhbX$ (or $\mathcal{RUB}^0_{h,b}$ in mixed characteristic) and $M_S^D$ be the minimal logarithmic structure on the base for the curve $D$ as an object of $\mathcal{M}^{ps-log}_{g,N}$. By pulling back $M_S^{Rub}$ by $\phi$ and pulling back $M_S^D$ by $\delta$, we can consider them to be sheaves of monoids on $S$. We know from \citep[Section 4.2]{chen2025tale} that the stalk of $\overline{M}_S^{Rub}$ at $s$ is the free monoid that is generated by the edges of slope $0$ and the negative levels. The stalk of $\overline{M}_S^D$ at $S$ is freely generated by the smoothing parameters of all nodes of $D$. So we have a morphism $M_S^D \oplus_{\mathcal{O}^\times_S}M_S^{Rub} \to N_S$. When $e$ is a node of $C$, so an edge in the dual graph of $C$ that maps by $f$ with multiplicity $n_e$ to a node $e'$ of $D$, then their smoothing parameters $\delta_e$ and $\delta_{e'}$ are linked by $n_e \delta_e = \delta_{e'}$. Let $H \subseteq \overline{N}_S^{gp}$ be the subgroup pointwise generated by these relations. We will now show that the logarithmic structure associated to the image of $\overline{M}_S^D \oplus \overline{M}_S^{Rub}$ in $\overline{N}_S^{gp} /H$ is the minimal monoid.

    We consider the object $(f': C' \to S', \ell_S')$ over the same underlying scheme $S$ and with characteristic monoid $\overline{M}_{S'}$ being associated to the image of $\overline{M}_S^D \oplus \overline{M}_S^{Rub}$ in $\overline{N}_S^{gp} /H$. We now have to show that in the diagram
    \begin{center}
    \begin{tikzcd}
      M_S  & M_{S''} \arrow[l] \\
       & M_{S'} \arrow[lu] \arrow[u, dotted]
    \end{tikzcd}
    \end{center}
    there exists the dotted arrow. But this is clear as by the arguments above the relations generating $H$ have to hold in $\overline{M}_{S''}$ as well.

    As we have explicitly constructed charts, the minimal logarithmic structure is coherent and by construction compatible with base change, this shows that $\PLHURphgnZ$ and $\PLHURphgnF$ are representable by algebraic stacks with coherent logarithmic structure. 

Our arguments can also be used to give objects of $\mathcal{LOG}\Mapfshgn$ minimal log structures: The logarithmic structure that is generated by a symbol for each edge of the dual graph where we have $n_1e_1=n_2e_2$ if $e_1$ and $e_2$ are mapped to the same same edge of the dual graph of the target curve with multiplicity $n_1$ and $n_2$ respectively. By the minimality criteria this gives us a logarithmic stack $\mathcal{L}{\Mapfshgn}^{log}$ parametrizing degree $p$ morphisms between log smooth curves that are either log étale or wild at each point and whose underlying stack is proper by the simultaneous semi-stable reduction theorem. Therefore, it follows that 
\begin{align*}
    \PLHURphgnZ &\subseteq \mathcal{RUB}^0_{h,b} \times_{\mathcal{M}_{h,b}^{log}} \mathcal{L}{\Mapfshgn}^{log} \\
    \PLHURphgnF &\subseteq \RubhbX \times_{\mathcal{M}_{h,b}^{log}} \mathcal{L}{\Mapfshgn}^{log}
\end{align*}
are strict closed substacks. As the fiber product of proper spaces is proper, we conclude that $\PLHURphgnZ$ and $\PLHURphgnF$ are proper. By construction the automorphism group scheme of an object of $\PLHURphgnZ$ or $\PLHURphgnF$ is a sub group scheme of $\RubhbX$ and $\mathcal{RUB}^0_{h,b}$ respectively, and as the latter stacks are Deligne-Mumford, so are $\PLHURphgnZ$ and $\PLHURphgnF$.
\end{proof}

The stacks we have just constructed are proper stacks that in the mixed characteristic case contain all admissible covers over a base of characteristic $0$ and in the equicharacteristic case all Hurwitz maps between smooth curves. Therefore, they are good candidates for a compactification. However, as we will see in the example in Section~\ref{sec:example}, they are too large as $\PLHURphgnZ$ can also contain irreducible components that generically are of positive characteristic and similarly, $\PLHURphgnF$ can contain components that generically parametrize covers of non smooth curves. This is similar to the compactification of the Moduli space of abelian differentials described in \cite{bainbridge2019moduli} where they cut out the component that parametrizes generically smooth curves by a global residue condition. In our case, this will be resembled by an exactness and quasi-exactness condition that we will formulate now.

For the following construction of bivariant differential forms, let $(f: C_S \to D_S, \ell_S)$ be an object of $\PLHURphgnZ$ or $\PLHURphgnF$ such that $\pi: C_S\to S$ is a nuclear log curve. Note that by Definition~\ref{def: enhanced_graph} and Lemma~\ref{lem: rescaling_ensamble} from this data, we get an enhanced level graph $\Gamma$ and a rescaling ensemble $s_i$. For each level $i$, let $C_{\leq i}$ be the curve over the closed stratum that is obtained by removing all irreducible components of level greater than $i$. We furthermore choose open subsets $U_i$ of $C_S$ that are open neighbourhoods of $C_{\leq i}$ in $C$ such that their collection satisfies $U_j \subset U_i$ if $j<i$. For example, one can take the $U_i$ to be complements of the closure of higher levels.

We observe that $\omega^{log}_f$ has a distinguished section:

\begin{definition}
    Let $(f: C_S \to D_S, \ell_S)$ be an object of $\PLHURphgnZ$ or $\PLHURphgnF$. The trace form $\tau_f \in H^0(\omega_f(-\ell_S))$ is the section coming from the structure morphism
    \begin{align*}
        f^* \Omega_{D/S}^{log} \to \Omega^{log}_{C/S} .
    \end{align*}
    This makes sense, because
    \begin{align*}
        \omega_f\coloneqq (f^*\Omega_D^{log})^\vee \otimes_{\mathcal{O}_C} \Omega^{log}_C \cong \mathcal{HOM}(f^*\Omega_D^{log}, \Omega^{log}_C).
    \end{align*}
\end{definition}

As described in \citep[Section 5.4.]{chen2025tale} by using a log splitting, one can construct out of an isomorphism 
\begin{align*}
    \omega_f^{log} \to \mathcal{O}_C(\ell_S)
\end{align*}
a collection of non-zero sections
\begin{align*}
\psi_i \in H^0(U_i, \omega^{log}_f(-\ell_S))
\end{align*}
that satisfy for each level $j,i$ with $j<i$
\begin{align*}
    \psi_i |_{U_i \cap U_j} = s_j \dots s_{i-1} \psi_j|_{U_i \cap U_j}
\end{align*}
where the $s_i$ are as in Lemma~\ref{lem: rescaling_ensamble}.

Choosing a different isomorphism $\omega_f^{log} \to \mathcal{O}_C(\ell_S)$ corresponds to rescaling each $\psi_i$ by a unit coming from the base, so in particular there is a unique isomorphism that induces on $U_0$ the trace form $\tau_f$.

\begin{remark}
    The trace form vanishes at the generic point of a component iff the map is separable there. Moreover, because of the stability condition, at components where the morphism is the relative Frobenius, the trace form has to have a pole when the orders are counted logarithmically, so with respect to $(dy/y) \otimes (dx/x)^{\vee}$ where $x \mapsto y^p$. In particular, we can conclude that the morphism $f$ of an object of $\PLHURphgnZ$ or $\PLHURphgnF$ is separable precisely at those components of maximal level.
\end{remark}

\begin{definition}
    Let $(f_S: C_S \to D_S, \ell_S)$ be an object of $\PLHURphgnZ$ or $\PLHURphgnF$ such that $\pi: C\to S$ is a nuclear curve. The collection of generalized relative differentials of $f$ is the collection $\psi_i \in H^0(U_i, \omega^{log}_f(-\ell_S))$ coming from that isomorphism $\omega_f^{log} \to \mathcal{O}_C(\ell_S)$ inducing $\tau_f$ on $U_0$.
\end{definition}
We finally have the necessary framework to define our Moduli functor by a condition similar to that used in \cite{bainbridge2019moduli} to cut out the locus of differentials on smoothable curves out of the whole space of families of generalized simple multi-scale differentials:

\begin{definition}
    Let $(f_S: C_S \to D_S, \ell_S)$ be an object of $\PLHURphgnZ$ or $\PLHURphgnF$. The family is said to satisfy the exactness and quasi-exactness condition if each point $s$ of $S$ has a nuclear neighbourhood $S'$ with $s$ in the closed stratum, such that the pullback of $(f_S: C_S \to D_S, \ell_S)$ to $S'$ satisfies the following property:
    
    For each negative level $i$ of the curve $C_{S'}$, the differential form $\psi_i$ is exact at all smooth points of the closed fiber of level $i$ if $s_i \neq p \in \mathcal{O}_{S'}/\mathcal{O}_{S'}^\times$ and quasi-exact at all smooth points of the closed fiber of level $i$ if $s_i = p \in \mathcal{O}_{S'}/\mathcal{O}_{S'}^\times$. In the mixed characteristic case of $\PLHURphgnZ$ we further require that there are no levels below the level of $p$, so there are no $s_j$ with $s_j >p$.
\end{definition}
This will be condition that cuts our desired Moduli stacks out of $\PLHURphgnZ$ and $\PLHURphgnF$:
\begin{definition}
    The Moduli functors $\LHURphgnZ$ and $\LHURphgnF$ are those subfunctors of $\PLHURphgnZ$ and $\PLHURphgnF$ respectively that are cut out by the exactness and quasi-exactness condition.
\end{definition}

These are indeed Moduli functors, so the functor maps morphisms to pullback diagrams, because we can pull back étale neighborhoods.
\begin{remark}
    An object of $\LHURphgnZ$ or $\LHURphgnF$ over an algebraically closed field of characteristic $p$ can be considered as a morphism together with a $p$-enhancement in the sense of \citep[Section 5.3.]{brezner2017lifting}. 

   As $\LHURphgnF$ is defined over $\mathbb{F}_p$, $p$ is zero so cannot be a rescaling parameter, and we conclude that in the equicharacteristic case quasi-exact forms do not appear at all.
\end{remark}
\begin{theorem}
    The Moduli functors $\LHURphgnZ$ and $\LHURphgnF$ are subfunctors of $\PLHURphgnZ$ and $\PLHURphgnF$ respectively, that contain the locus of covers of smooth curve. Their points in characteristic $p$ correspond to those covers that can be lifted to covers of smooth curves.
\end{theorem}
\begin{proof}
In this proof, we will always assume that the objects we consider have the minimal log structure described in the proof of Theorem~\ref{thm:plhur_rep}. We start by showing that the closure of the locus of covers between smooth curves in $\PLHURphgnZ$ and $\PLHURphgnF$ is contained in $\LHURphgnZ$ and $\LHURphgnF$ respectively. 
\paragraph*{\underline{Closure of Locus of Smooth Covers: Determination of $s_i$ and $\psi_i$}}
Let $R$ be a discrete valuation ring and $(f_S: C_S \to D_S, \ell_S)$ be an object of $\PLHURphgnZ$ or $\PLHURphgnF$ respectively, such that the fiber over the generic point of $R$ is a cover between smooth curves. We will show that then $(f_S: C_S \to D_S, \ell_S)$ already is an object of $\LHURphgnZ$ or $\LHURphgnF$ respectively. We have seen that the admissible cover space is the generic fiber of $\PLHURphgnZ$ which is known to be proper. So we can assume that when considering $\LHURphgnZ$ the discrete valuation ring $R$ is of mixed characteristic $(0,p)$ and in the case of $\LHURphgnF$ the ring $R$ will be of equicharactistic $p$. The section $\ell_S$ determines the rescaling ensemble $s_i$ and the collection of rescaled bivariant differentials $\psi_i$. So all we have to check is that the rescaling ensemble and the collection of rescaled bivariant differentials satisfy the exactness and quasi-exactness condition.
\paragraph*{\underline{Closure of Locus of Smooth Covers: Description of $f$}}
To see this, let $q$ be a smooth point of the special fiber of $C_S$ at a non-maximal level $i$. Then locally around $q$ the morphism $f$ is given by a ring map $R[x] \to T$ with the ring $T$ étale over $R[y]$ for some $y$. As $f$ is inseparable at $q$, the parameter $y$ can be chosen such that $f$ is given by $x \mapsto y^p+m h$, where $m$ is an element from the maximal ideal of $R$ and $h$ is some function in $T$. Note that we can choose the coordinates in such a way, such that $d_Ch \neq 0$ or $h=0$, and $h=0$ can only appear in the mixed characteristic case, as else the morphism was inseparable also at the generic fiber contradicting our assumption.
\paragraph*{\underline{Closure of Locus of Smooth Covers: Exactness and Quasi-Exactness}}
From the last description we know that locally, the trace form is given by $\tau_f= (d_D x)^\vee  \otimes (p y^{p-1}d_Cy +m d_Ch )$. In the equicharacteristic case $p y^{p-1}=0$, the rescaling parameter $s_i \in R$ is of the same order as $m$ and the form $\psi_i$ is $(d_D x)^\vee  \otimes ( d_Ch )$, so is exact. In the mixed characteristic case, we see that $s_i$ is of order at most $p$, as else $\psi_i$ would have a pole along the whole component because of the term consisting of $py^{p-1}$. The summand $y^{p-1}$ is nonzero after reduction iff $s_i$ is of the same order as $p$. In this case, the form we obtain is quasi-exact. If $s_i$ is of order less than $p$, the form is exact as desired because the term containing $py^{p-1}$ vanishes. We have seen that the semistable model of a finite cover between smooth curves satisfies the exactness and quasi-exactness condition. So the closure of the locus of covers between smooth curves is contained in $\LHURphgnZ$ or $\LHURphgnF$ respectively.
\paragraph*{\underline{Lifting: Setup}}
Now we want to show that any point of $\LHURphgnZ$ or $\LHURphgnF$ is in the closure of $\PLHURphgnZ$ or $\PLHURphgnF$ respectively. Our main ingredient will be the lifting result from \citep[Theorem 6.2.3.]{brezner2017lifting} Let $r$ be such a $k$-point with $k$ of characteristic $p$ corresponding to an object $(f_k: C_k \to D_k, \ell_k)$. Let $K$ be a complete algebraically closed valued field that in the case of $\LHURphgnZ$ is of mixed characteristic, in the case of $\LHURphgnF$ of equicharacteristic with ring of integers $R$ and such that the residue field is a field extension of $k$. Such a $K$ can be constructed the following way: In the mixed characteristic case, one can take the ring of Witt vectors $W(\overline{k})$ over the algebraic closure $\overline{k}$ of $k$ and take the completion of the algebraic closure of its fraction field. In the equicharacteristic case, one can take $K$ to be the completion of the algebraic closure of the fraction field of $\overline{k}[[t]]$.
\paragraph*{\underline{Applying the Analytic Lifting Theorem}}
We can pull back $r$ to the residue field $l$ of $R$ and after pulling back the log structure further, we get a morphism between nice log curves in the language of \cite{brezner2017lifting}, the piecewise linear function $\ell_k$ together with the collection of rescaled bivariant differentials constitutes a $p$-enhancement in the sense of \citep[5.3.3.]{brezner2017lifting}. Therefore, by \citep[Theorem 6.2.3.]{brezner2017lifting} there is a finite cover $f_K$ of $K$-analytic Berkovich smooth curves that induces $f_l$ as its special fiber and $\ell_l$ as PL function in the log reduction. 
\paragraph*{\underline{Algebraization}}
As the Berkovich analytification functor is an equivalence of categories between projective algebraic and projective analytic varieties, we know that there is a degree $p$ cover between smooth projective algebraic curves over $K$ inducing $f_K$ under analytification that we will call $f_K$ as well. Because of properness, this uniquely extends to an object of $\PLHURphgnZ$ or $\PLHURphgnF$ respectively over $Spec(R)$ with closed fiber being $(f_l: C_l \to D_l, \ell_l)$ as the minimal formal model of $f_K$ coming from Berkovich analytic geometry has the same special fiber as the algebraic stable model of $f_K$ as the formal model is the completion of the algebraic model \citep[Lemma 5.1.]{amini2015lifting}. 

Recall $(f_l: C_l \to D_l, \ell_l)$ is pulled back from $r$. As the generic fiber is a cover of smooth curves, we know that $r$ is in the closure of the locus of $\PLHURphgnZ$ or $\PLHURphgnF$ respectively that parametrizes covers between smooth curves.
\end{proof}

    As from this description it follows properness, we can immediately conclude the following corollary:
    \begin{corollary}
        The underlying stacks of $\LHURphgnZ$ and $\LHURphgnF$ are proper Deligne-Mumford stacks.
    \end{corollary}
    In order to analyze $\LHURphgnZ$ and $\LHURphgnF$ further, we will need the following lemma that allow us to describe the minimal log structures of their objects.
    \begin{lemma} \label{cor:min_log_stru_cov}
        The minimal logarithmic structure on the base for a cover $f: C \to D$ in the stack $\LHURphgnZ$ in characteristic $p$ is a monoid of rank $\#E_D^{hor}+\#V_C^{ex}+1$, where $\#E_D^{hor}$ is the number of edges of the dual graph of $D$ that are in the image of edges of slope $0$ of $C$ and $\#V_C^{ex}$ is the number of irreducible components of $C$ with an exact differential form. In the equicharacteristic case, so for a cover $f: C \to D$ in $\LHURphgnF$ this minimal monoid is of rank $\#E_D^{hor}+\#V_C^{ex}$. If $p=2$, these monoids are free.
    \end{lemma}
    \begin{proof}
        Smoothing parameters of the nodes of slope $0$ of $C$ that map to the same node of $D$ are identified by $f$. So the smoothing parameters of the nodes of slope $0$ contribute $\#E_D^{hor}$-many generators. If $p=2$ and there is a relation between the smoothing parameters $\delta_1$ and $\delta_2$ of two nodes or markings of $C$, the map is étale at both of them, so the smoothing paramaters get identified. Therefore, if $p=2$ the minimal monoid is free.
        
        For $p>2$ it might happen that nodes of the source curve map to the same node of the target curve with different multiplicities, which leads to the monoid being non free. As there are no relations between the levels of irreducible components with an exact differential form, there are $\#V_C^{ex}+1$ copies of $\N$ that correspond to negative levels. The only difference between the mixed and equicharacteristic case is that the level $log(p)$ that is pulled back from the base does not appear in the equicharacteristic case.
    \end{proof}
    \subsection{First Properties of \texorpdfstring{$\LHURphgnZ$}{LHURphgnZ}}
    In this subsection, we will consider the mixed characteristic case. We want to analyze some properties of the underlying stack of $\LHURphgnZ$ that follow directly from its construction. At first, note that flatness over a discrete valuation ring is equivalent to torsion freeness, so we can conclude that a scheme is flat over a discrete valuation ring iff its generic fiber is schematically dense. We will formulate this as the following lemma:
    \begin{lemma}
        Let $(R,\mathfrak{m)}$ be a discrete valuation ring with fraction field $K$ and $Y\to Spec(R)$ be a scheme over $R$. Let $X \subseteq Y_\eta$ be a closed subscheme of the generic fiber of $Y$ and $\overline{X} \subseteq Y$ be its scheme theoretic closure in $Y$. Then $\overline{X}$ is flat over $R$.
    \end{lemma}
   
    \begin{theorem} \label{thm:equ-dim}
        All irreducible components of the special fiber of $\LHURphgnZ$ are of dimension $3g-3+N$.
    \end{theorem}
    \begin{proof}

In the last section we saw that the reduction of the underlying stack of $\LHURphgnZ$ is the reduction of the closure of the admissible cover stack in the underlying stack of $\PLHURphgnZ$, so in particular is flat over $\mathbb{Z}_{(p)}$. Recall that the admissible cover stack is of dimension $3g-3+N$. The statement now follows from \citep[Theorem 14.116]{gortz2020algebraic} stating that for a proper flat family of schemes with equidimensional source one can conclude that each fiber is equidimensional if one knows that the generic fiber is equidimensional and in this case, also the dimension of the fibers agree.
    \end{proof}
\section{Moduli Spaces of Relevant Objects} \label{sec:rel_obj}
Having defined our objects of interest $\LHURphgnZ$ and $\LHURphgnF$, we now want show evidence that these are indeed interesting objects. In particular, we will show that $\LHURphgnZ$ and $\LHURphgnF$ are log smooth under the technical restrictions $p=2$ and $g=0$. These cases already include the example of degree two covers of the projective line by an elliptic curve that was described in \cite{abramovich1998stable} but could so far not been explained conceptually. To do so, in this section we introduce various Moduli spaces that will parametrize some components of our Hurwitz space and study their properties.

\subsection{The Moduli Space  \texorpdfstring{$\AScovhen$}{AScovhen} of Artin--Schreier Covers with Fixed Ramification Divisor}
\begin{definition}
    Let $\vec{e} =(e_1, \dots e_m) \in \Z_{>0}^m$ be a vector such that $e_i \neq 1 \mod p$ for all $i$ and such that
    \begin{align*}
        \sum_{i=1}^m e_i = \frac{2h}{p-1}+2.
    \end{align*}
    An $n$-marked Artin--Schreier cover with conductors $\vec{e}$ over an $\mathbb{F}_p$-scheme $S$ consists of a smooth marked curve $C$ of genus $h$ with $m+pn$ markings over $S$, so an object of $\mathcal{M}_{h,m+pn}$ over $S$ and a $\mathbb{Z}/p\mathbb{Z}$-action on $C$ such that the quotient $f: C \to C/(\mathbb{Z}/p\mathbb{Z}) \cong 
    \PjS$ is a finite separable morphism of degree $p$ over $S$ such that the ramification divisor of $f$ as in \citep[Proposition 3.3.]{pries2012p} is given by
    \begin{align*}
        R(f) = \sum_{i=1}^m e_i(p-1) [f^{-1}(p_i)]
    \end{align*}
    and such that the markings that do not correspond to ramified points map in an ordered way to $n$ marked points in the target. 

    The Moduli functor $\AScovhen$ maps an $\mathbb{F}_p$-scheme $S$ to the set of $n$-marked Artin--Schreier covers with conductors $\vec{e}$ over $S$.
\end{definition}
\begin{remark} \label{rem:con_pole}
     The condition on the ramification divisor is equivalent to the trace form $\tau_f$ having logarithmic zeroes of order $(e_j-1)(p-1)$ along the $j$-th marking.
    
    Furthermore, note that one can think about the additional $np$ markings as the preimage of $n$ markings of the target curve $\Pj$.
\end{remark}
\begin{remark}
This definition implies by \citep[Section 3.3.]{pries2012p} that after a finite flat extension of $S$ and symbolically writing $x = (x-\infty)^{-1}$, one can express $f$ as 
\begin{align*}
    y^p-y = \sum_{i=1}^m h_i\left(\frac{1}{x-f(p_i)}\right)
\end{align*}
with each $h_i$ being a polynomial of degree $e_i-1$.
\end{remark}
\begin{theorem} \label{thm: as_smooth}
The Moduli space of $\AScovhen$ is smooth of dimension
\begin{align*}
\frac{2h}{p-1}+n-1-\sum_{i=1}^{m} \left \lfloor {\frac{e_i-1}{p}} \right \rfloor = n+m-3+\sum_{i=1}^{m} \left(e_i-1- \left \lfloor {\frac{e_i-1}{p}} \right \rfloor \right)
\end{align*}
\end{theorem}
\begin{proof}
Representability holds as the space of $\mathbb{Z}/p\mathbb{Z}$-actions on smooth curves of genus $g$ such that the quotient is of genus $0$ is representable by a stack of finite type over $\mathbb{F}_p$ due to \citep[p.1]{maugeais2005compactification} and the condition on the ramification indices cuts out $\AScovhen$ as a locally closed substack.

We set for each $1 \leq i \leq m$
\begin{align*}
t_i=e_i-1-\left \lfloor {\frac{e_i-1}{p}} \right \rfloor \\
M_i \coloneqq \G_a^{t_i-1} \times \G_m
\end{align*}
and define
\begin{align*}
M=\prod_{i=1}^{m} M_i.
\end{align*}
One can write $f$ as in the last remark, and as described in \citep[Section 3.3]{pries2012p} after a finite flat base change $S' \to S$ to an affine $S'=Spec(A)$ the pullback of the cover to $S'$ can be written uniquely as
\begin{align*}
y^p-y=\sum_{i=1}^{m} \sum_{k=1}^{e_i-1} a_{k,i} \frac{1}{(x-f(p_i))^k}
\end{align*}
where $a_{k,i} \in A$ is nilpotent whenever $p$ divides $k$ and $a_{e_i-1,i} \neq 0$. This can be achieved by a change of coordinates that allows us to include $p$-th powers in $x$ into the $y^p-y$ term. For details, see \citep[Remark 3.9.]{pries2012p}. Additionally, we keep track of $n$ markings $q_1, \dots, q_n$ on $\PjS$ away from the images of the ramified points.

Two such marked covers are isomorphic iff there is an automorphism of $\mathbb{P}^1_{S'}$ fixing the location of the $f(p_i),q_j $ and transforming one cover into the other. In particular, when there are at least three marked points, two covers are isomorphic if all the coefficients $a_{k,i}$ agree. So we conclude that for a fixed divisor $D=\sum e_iq_i$ the space of $n$-marked Artin--Schreier covers $f$ with ramification divisor
     \begin{align*}
        R(f) = \sum_{i=1}^m e_i(p-1) [f^{-1}(p_i)]
    \end{align*}
is given by $M$ together with the choice of $n$ markings away from the ramified points on $\PjS$.
%By \citep[Section 5.1]{bertin2000deformations} we can now conclude that this is also true when we deform the divisor but do not collide points. 

We have seen that each point of $S \times_{\mathbb{F}_p} \AScovhen$ has a fppf neighbourhood $S'$ such that the pullback $S' \times_{\mathbb{F}_p} \AScovhen$ is locally isomorphic to $\mathcal{M}_{0,n+m} \times M$ there. As smoothness is fppf local on the base and
\begin{align*}
n+m-3+\dim(M)=\frac{2h}{p-1}+n-1-\sum_{i=1}^{m} \left \lfloor {\frac{e_i-1}{p}} \right \rfloor = n+m-3+\sum_{i=1}^{m} \left(e_i-1- \left \lfloor {\frac{e_i-1}{p}} \right \rfloor \right)
\end{align*}
by trivial arithmetic, it follows the statement.
\end{proof}
\subsection{The Moduli Space of Exact Differential Forms \texorpdfstring{$\Mnexm$}{M0nexm}}
For the rest of this subsection, we fix an integer partition $\mathbf{m}=(m_1, \dots, m_n)$ of $2p-2$ of length $n$.
\begin{definition}
The Moduli space $\Mnexm \subseteq \Mn$ parametrizes those configurations of the markings $(p_1, \dots, p_n)$ on the projective line such that there is an exact non-zero meromorphic section $\omega$ of the sheaf $\Omega_{{\PjS/S}}^{(1-p)}$ considered as the dualizing sheaf of the relative Frobenius with zeroes and poles according to $\mathbf{m}$.
\end{definition}

\begin{theorem} \label{thm:exc_smooth}
    The Moduli space $\Mnexm$ is smooth over $\mathbb{F}_p$ and has dimension
    \begin{align*}
        dim(\Mnexm) = n-4+\sum_{i=1}^n \left \lfloor{\frac{m_i}{p}} \right \rfloor
    \end{align*}
    or is empty.
\end{theorem}

\begin{proof}
    To see that $\Mnexm$ is a closed substack of $\mathcal{M}_{0,n}$, note that it is cut out by the vanishing locus of a twisted Cartier operator by Theorem~\ref{thm: rel_car_prop}. So similar to \cite{polishchuk2006moduli} it can be constructed as the degeneracy locus of a map of vector bundles on $\mathcal{M}_{0,n}$.

    From considering the ranks of these vector bundles, one gets

 \begin{align*}
        dim(\Mnexm) \geq n-4+\sum_{i=1}^n \left \lfloor{\frac{m_i}{p}} \right \rfloor
    \end{align*}
    from  \cite{geertsen2002degeneracy}.
    
    We will assume that all markings are away from $\infty$, this can always be achieved by a Möbius transformation. Let $(C,p_1, \dots, p_n)$ be a $k$-point of $\Mnexm$, for $k$ an algebraic closure of $\mathbb{F}_p$. From our discussion of the Cartier operator we know that one can write 
    \begin{align*}
        \omega= \prod_{i=1}^n (y-p_i)^{m_i}\frac{dy}{dx} =\sum_{i=0}^{p-1} f_i^p y^i \frac{dy}{dx}    
    \end{align*}
    for some rational functions $f_i$ with $f_{p-1}=0$. Because of the linearity of the twisted Cartier operator, this is in fact equivalent to 
    \begin{align*}
        \rc\left(\prod_{i=1}^n (y-p_i)^{m_i-p \left \lfloor{\frac{m_i}{p}} \right \rfloor}\frac{dy}{dx}\right)=0.
    \end{align*}
    To simplify notation, we set 
    \begin{align*}
        m'_i& \coloneqq m_i-p\left \lfloor \frac{m_i}{p} \right \rfloor
    \end{align*}
    and note that $0 \leq m'_i < p$. Those $i$ where $m'_i=0$, so $p  | m_i$, do not appear in this expression. The set of these indices $i$ were $m_i$ is divisible by $p$ will be denoted by $I_p$.

    A tangent vector to $(\Pj, p_1, \dots, p_n)$ in $\Mn$ is an element of $H^1(\Omega_{\Pj}^\vee(-\sum p_i))$. It is defined by how it deforms the markings, so it can be considered as an element of $k^{n-3}$ by the isomorphism
    \begin{align*}
        (a_1, \dots, a_{n-3}) \mapsto (\mathbb{P}_{k[\epsilon]}, p_1 +a_1\epsilon, \dots, p_{n-3}+a_{n-3}\epsilon, p_{n-2}, p_{n-1}, p_n).
    \end{align*}
    Such a tangent vector now is also a tangent vector to $\Mnexm$ iff the twisted differential form on $\mathbb{P}^1_{k[\epsilon]}$
    \begin{align*}
        \prod_{k=0}^2(y-p_{n-k})^{m_{n-k}}\prod_{i=1}^{n-3} (y-(p_i+a_i\epsilon))^{m_i}\frac{dy}{dx}
         \end{align*}
   is exact. We can rewrite this differential form to be
   \begin{align*}
       \omega+\epsilon \left(\prod_{k=0}^2(y-p_{n-k})^{m_{n-k}}\prod_{\substack{l=1}}^{n-3}(y-p_l)^{m_l-1}\left(\sum_{\substack{i=1 \\ i \notin I_p}}^{n-3} a_i m'_i \prod_{\substack{j=1 \\ j \notin I_p \cup \{i \} }}^{n-3} (y-p_j)\right)\frac{dy}{dx}\right)
   \end{align*}
   As we can factor out $p$-th powers when checking exactness, exactness of this form is equivalent to
    \begin{align*}
       \mathfrak{tc}\left(\prod_{k=0}^2(y-p_{n-k})^{m'_{n-k}}\prod_{\substack{l=1 \\ l \notin I_p}}^{n-3}(y-p_l)^{m'_l-1}\left(\sum_{\substack{i=1 \\ i \notin I_p}}^{n-3} a_i m'_i \prod_{\substack{j=1 \\ j \notin I_p \cup \{i \} }}^{n-3} (y-p_j)\right)\frac{dy}{dx}\right) &=0.
    \end{align*}
    As the products $\prod_{\substack{j=1 \\ j \neq i }} ^n (y-p_j)$ are linearly independent, we have an epimorphism
    \begin{align*}
        \alpha: H^1(\Omega_{\Pjk}^\vee(-\sum_{i=1}^n[p_i] )) \to H^0(((F_{\Pjk})_*\Omega_{F,\Pjk})(-D))
    \end{align*}
    for
    \begin{align*}
    D&=\sum_{k=0}^2 m'_{n-k}[p_{n-k}]+\sum_{\substack{i=1 \\ i \notin I_p}}^{n-3}(m'_i-1)[p_i] - \left(\left(\sum_{i=1}^n m'_i\right)-2p+1\right)[\infty] \\
    &=\sum_{k=0}^2 m'_{n-k}[p_{n-k}]+\sum_{\substack{i=1 \\ i \notin I_p}}^{n-3}(m'_i-1)[p_i] + \left(p\left(\sum_{i=1}^n \left \lfloor \frac{m_i}{p} \right \rfloor\right)+1\right)[\infty]
    \end{align*}
    that maps
    \begin{align*}
        (a_1, \dots, a_{n-3}) \mapsto \prod_{k=0}^2(y-p_{n-k})^{m'_{n-k}}\prod_{\substack{l=1 \\ l \notin I_p}}^{n-3}(y-p_l)^{m'_l-1}\left(\sum_{\substack{i=1 \\ i \notin I_p}}^{n-3} a_i^p m'_i \prod_{\substack{j=1 \\ j \notin I_p \cup \{i \} }}^{n-3} (y-p_j)\right)\frac{dy}{dx}
    \end{align*}
    The kernel of $\alpha$ consists exactly of those deformations that only deform markings $p_i$ with $i \in I_p$. In particular,
    \begin{align*}
        dim(ker(\alpha))= \#I_p.
    \end{align*}
We have seen that a tangent vector $\nu \in H^1(\Omega^\vee_{\Pj}(-\sum_{i=1}^n p_i))$ is also a tangent vector to $\Mnexm$ iff the image of $\nu$ under $\alpha$ is in the kernel of the twisted Cartier operator on global sections:
\begin{align*}
    H^0(\mathfrak{tc}): H^0(((F_{\Pjk})_*\Omega_{F,\Pjk})(-D)) \to H^0(\mathcal{O}_{\Pjk}\left( -\sum_{i=1}^n\left \lfloor \frac{m_i}{p} \right \rfloor [\infty]\right))
\end{align*}
From Theorem~\ref{thm: rel_car_prop} we know that that $H^0(\mathfrak{tc})$ is surjective, so we can compute the kernel to be of dimension:
\begin{align*}
    \dim(\ker(H^0(\mathfrak{tc})))&=h^0((Frob_*\Omega_{F,\Pjk})(-D))-h^0\left(-\sum_{i=1}^n\left \lfloor \frac{m_i}{p} \right \rfloor [\infty]\right) \\
    &=h^0(\Omega_{F,\Pjk}(n-3-\#I_p-(2p-2)-1))- h^0\left(-\sum_{i=1}^n\left \lfloor \frac{m_i}{p} \right \rfloor [\infty]\right)\\ 
    &=h^0(\mathcal{O}_{\Pjk}(n-4-\#I_p)) - h^0\left(-\sum_{i=1}^n\left \lfloor \frac{m_i}{p} \right \rfloor [\infty]\right) \\
    &= n-3 -\#I_p -\left(\sum_{i=1}^n\left(-\left \lfloor \frac{m_i}{p} \right \rfloor \right)+1\right) \\ 
    &=n-4- \#I_p+\sum_{i=1}^n \left \lfloor \frac{m_i}{p} \right \rfloor
\end{align*}
We have identified the tangent space $T_{(\Pjk,p_1, \dots,p_n)} \Mnexm$ to be the kernel of the composition
 \begin{center}
    \begin{tikzcd}
        H^1(\Omega_{\Pjk}^\vee(-\sum_{i=1}^n p_i) \arrow[r, " H^0(\rc) \circ \alpha"]  &  H^0(\mathcal{O}_{\Pjk}( -\sum_{i=1}^n\left \lfloor \frac{m_i}{p} \right \rfloor [\infty]))
    \end{tikzcd}
    \end{center}
    and have seen that $\alpha$ is surjective, so we can compute
    \begin{align*}
        \dim(T_{(\Pjk,p_1, \dots,p_n)} \Mnexm) &= \dim(\ker(H^0(\mathfrak{tc}) \circ \alpha)) \\&= \dim(\ker(\alpha))+ \dim(\ker(H^0(\mathfrak{tc})) = n-4+\sum_{i=1}^n \left \lfloor \frac{m_i}{p} \right \rfloor.
    \end{align*}
    As this is independent of $(\Pjk, p_1, \dots, p_n)$ we see that $\Mnexm$ is smooth of the desired dimension if it is non-empty.
    \end{proof}

    \begin{remark} \label{rem: tan-sur}
        The proof implies in particular that $H^0(\mathfrak{tc}) \circ \alpha$ is surjective, this will be helpful when we consider quasi-exact forms.
    \end{remark}
\subsection{The Moduli Space of Quasi-Exact Differential Forms \texorpdfstring{$\Mnquexm$}{M0nquexm}}
Again, we fix an integer partition $\mathbf{m}=(m_1, \dots, m_n)$ of $2p-2$ of length $n$.
\begin{definition}
The Moduli space $\Mnquexm \subseteq \Mn$ parameterizes those configurations of the markings $(p_1, \dots, p_n)$ on the projective line such that there is a quasi-exact meromorphic section $\omega$ of the sheaf $\Omega_{{\PjS/S}}^{(1-p)}$ with zeroes and poles according to $\mathbf{m}$.
\end{definition}

\begin{theorem}
    The Moduli space $\Mnquexm$ is smooth over $\mathbb{F}_p$ and has dimension
    \begin{align*}
        dim(\Mnquexm) = n-3+\sum_{i=1}^n \left \lfloor{\frac{m_i}{p}} \right \rfloor
    \end{align*}
    or is empty.
\end{theorem}
 \begin{proof}
     Again, by Theorem~\ref{thm: rel_car_prop}, $\Mnquexm$ is a closed substack of $\mathcal{M}_{0,n}$. The techniques in this proof are very similar to those that we used in the corresponding statement for exact forms. 
     
     By considering the ranks of the vector bundle whose degeneracy locus in $\Mnquexm$, we get
      \begin{align*}
        dim(\Mnexm) \geq n-3+\sum_{i=1}^n \left \lfloor{\frac{m_i}{p}} \right \rfloor
    \end{align*}Again, we can assume that all markings are away from $\infty$. Let $(\Pjk,p_1, \dots, p_n)$ be a $k$-point of $\Mnquexm$ for $k$ an algebraic closure of $\mathbb{F}_p$. This means that
     \begin{align*}
        \mathfrak{tc} \left(\prod_{i=1}^n (y-p_i)^{m_i}\frac{dy}{dx} \right) \in k^\times.
     \end{align*}
    We define
     \begin{align*}
         G=\prod_{i=1}^n (y-p_i)^{\left \lfloor{\frac{m_i}{p} } \right \rfloor} 
     \end{align*}
     and again
    \begin{align*}
        m'_i& \coloneqq m_i-p\left \lfloor \frac{m_i}{p} \right \rfloor
    \end{align*}
    one sees that the condition is equivalent to 
    \begin{align*}
        \mathfrak{tc}\left(\prod_{i=1}^n(y-p_i)^{m_i'}\frac{dy}{dx}\right)  = cG
    \end{align*}
    for a $c \in k^\times$. We again identify $ H^1(\Omega_{\Pjk}^\vee(-\sum_{i=1}^n p_i))$ with $k^{n-3}$ and define the set $I_p$, the divisor $D$ and the maps $\alpha$ and $H^0(\mathfrak{c})$ as in the proof of Theorem~\ref{thm:exc_smooth}. By the same techniques as in this proof, a tangent vector $\nu \in T_{(\Pjk, p_1, \dots, p_n)} \Mn$ is a tangent vector to $\Mnquexm$ iff its image under $ H^0(\mathfrak{tc}) \circ \alpha$ is in the $k$-vector space generated by $G$. So we can identify the tangent space $T_{(\Pjk, p_1, \dots, p_n)} \Mnquexm$ with the kernel of the map
    \begin{align*}
        \pi \circ H^0(\mathfrak{tc}) \circ \alpha: H^1(\Omega_{\Pjk}^\vee(-\sum_{i=1}^n p_i)) \to H^0\left(\mathcal{O}_{\Pjk}\left( -\sum_{i=1}^n\left \lfloor \frac{m_i}{p} \right \rfloor [\infty]\right)\right)/  G \Gamma(\mathcal{O}_{\Pjk}).
    \end{align*}
    where $\pi$ denotes the projection.
    From Remark~\ref{rem: tan-sur} we know that $H^0(\mathfrak{c}) \circ \alpha$ is surjective, so also $\pi \circ H^0(\mathfrak{tc}) \circ \alpha$ is surjective and we can compute
    \begin{align*}
        \dim(T_{(\Pjk,p_1, \dots,p_n)} \Mnquexm) &= n-3-\left(\left(-\sum_{i=1}^n \left \lfloor{\frac{m_i}{p}} \right \rfloor +1\right)-1\right) = n-3+\sum_{i=1}^n \left \lfloor{\frac{m_i}{p}} \right \rfloor.
    \end{align*}
    So the tangent space of $\Mnquexm$ is at every point of the same dimension, so $\Mnquexm$ is smooth and of the dimension of its tangent space or empty.
 \end{proof}
 
\section{Logarithmic Smoothness of \texorpdfstring{$\LHURphgnZ$}{LHZA} and \texorpdfstring{$\LHURphgnF$}{LHFAXi} for \texorpdfstring{$p=2$}{p=2} and \texorpdfstring{$g=0$}{g=0}} \label{sec:log_smooth}
The objective of this section is to prove our main theorem. Recall that $A$ abbreviates the data $(h,g,N,\Lambda)$.
\begin{theorem} \label{thm:log_hur_log_smooth}
    For $p=2$ and $g=0$, the Hurwitz spaces $\LHURphgnZ$ and $\LHURphgnF$ are log smooth over $\Z_p$ with logarithmic structure coming from $\N$ that maps $1$ to $p$ and $\mathbb{F}_p$ with trivial logarithmic structure respectively.
\end{theorem}
\begin{remark}
The reason why we restrict this statement to $p=2$ is that in characteristic $2$ every separable cover is Galois, so given by an Artin--Schreier equation. We do not have yet an analogous result to Theorem~\ref{thm: as_smooth} for separable but non-normal covers that we would need to extend our proof also to odd primes. In the future, it would be interesting to investigate these Moduli spaces to extend our results to larger primes. Additionally, one can also consider the case of Galois covers by extending the reduction and lifting result from \cite{brezner2017lifting} to generalize our findings. As in these settings Lemma~\ref{lem: dim_calc} will be relevant for general $p$ and the proof is not more complicated, we will formulate it the general form.
\end{remark}
We will use the chart criteria for logarithmic smoothness from \citep[Theorem 3.5.]{kato1989logarithmic}:
\begin{theorem}
    Let $g: X \to Y$ be a morphism of fine logarithmic schemes and $x$ be a geometric point of the underlying scheme of $X$. Then $g$ is log smooth at $x$ iff there are étale neighbourhoods with charts $U \to \A_P$ of $x$, $V \to \A_Q$ of $g(x)$ and a morphism $\varphi: Q \to P$ such that 
   \begin{center}
    \begin{tikzcd}
        U \arrow[d, "g"] \arrow[r] & \A_P \arrow[d, "\A_\varphi"] \\
        V \arrow [r] & \A_Q
    \end{tikzcd}
    \end{center}
    commutes and
    \begin{enumerate}
        \item The base change $U \to V \times_{\A_Q} \A_P$ is smooth as a morphism of schemes.
        \item $\ker(\varphi^{gp})$ and $\text{coker}(\varphi^{gp})_{tor}$ are finite of order invertible in $k(x)$.
    \end{enumerate}
    \end{theorem}
We will give the proof for $\LHURphgnZ$, the proof for $\LHURphgnF$ is precisely the same dimension counting argument that will be even simpler, as no quasi-exact forms can appear.
The chart of the log structure of $\Z_{(p)}$ is given by 
\begin{align*}
    \Z[ \N log(p)] \to \Z_{(p)}
\end{align*}
where we consider $p$ to be the uniformizer of $\Z_{(p)}$.

 First, note that whenever $x \in \LHURphgnZ$ is a point over the generic fiber, we simply can take $U$ to to be the whole generic fiber of $\LHURphgnZ$ which we know from Theorem~\ref{thm: gen_fibre_smooth} to be log smooth.

So let $ x \in \LHURphgnZ$ be a point in the special fiber. We have seen in Corollary~\ref{cor:min_log_stru_cov} that after choosing a suitable étale neighborhood $U$ around $x$ the chart at $x$ is given by 
\begin{align*}
    U \to Spec(\Z[\N log(\rho) \oplus\N^{\#V_C^{ex}+\#E_D^{hor}}])
\end{align*}
where $\#V_C^{ex}$ denotes the number of components with an exact differential form of the source curve and $\#E_D^{hor}$ denotes the number of horizontal edges of the target curve $D$, so the cardinality of the image of the horizontal edges of slope $0$ of $C$. $g^*$ maps $p$ to $\rho$ that is the parameter corresponding to the minimal level. This holds because to the minimal level the parameter $log(p)$ was assigned. So the conditions on kernel and cokernel of $\varphi^{gp}$ are trivial and we are left to show that
\begin{align*}
    \gamma: U \to Spec(\Z_{(p)} \otimes _{\Z[\N log(p)]} \Z[\N log(\rho) \oplus\N^{\#V_C^{ex}+\#E_D^{hor}}]) = Spec(\Z_{(p)}[\N^{\#V_C^{ex}+\#E_D^{hor}}])
\end{align*}
is smooth.

Our main ingredient will be the following lemma.
\begin{lemma} \label{lem: dim_calc}
    The fiber $U_0$ of $\gamma: U \to Spec(\mathbb{Z}_{(p)}[\N^{\#V_C^{ex}+\#E_D^{hor}}])$ over the origin, determined by the vanishing of all coordinates and $p$, is smooth of dimension 
    \begin{align*}
        \dim(U_0)=N-3-(\#V_C^{ex}+\#E_D^{hor}).
    \end{align*}
\end{lemma}
\begin{proof}
    By construction $U_0$ is a connected component of the logarithmic stratification of $U$ supported over the special fiber of $\mathbb{Z}_{(p)}$. In particular, we know that the dual graph of the source $C$ and target curve $D$ of the universal family and the slopes of the PL function are constant along $U_0$. In \citep[Lemma 4.3.3.3]{marcus2020logarithmic} the authors describe the preimage of the Abel-Jacobi section. Note that they compute the preimage for the base being an algebraically closed field, but the only assumption they use is that the logarithmic structure on the base is constant, so their argument is applicable to our case as well. 
    
    As an Artin--Schreier cover is uniquely defined by the location of the ramified points and their conductors which are determined by the PL function, there is no Moduli along the nodes. Note that the conductors of the Artin--Schreier covers and the order of zeroes and poles are constant along the log stratification of $U$. Therefore, by the proof of \citep[Lemma 4.3.3.3]{marcus2020logarithmic} $U_0$ is isomorphic to an open neighborhood of a finite group quotient of a product of Moduli spaces of Artin--Schreier covers with fixed conductors and Moduli spaces of exact and quasi-exact differential forms with fixed order of zeroes and poles, where the quotient forgets the order of markings that are glued together to edges. Now, by the findings from the last section, $U_0$ is a finite quotient of a smooth stack, so is itself is smooth.

    To compute the dimension of $U_0$, we simply have to sum up the dimensions of each component that we computed in the last section. We begin with the maximal level where the covers are given by Artin--Schreier covers.

    In Remark~\ref{rem:con_pole}, we have linked the conductors with the orders of zeroes of bivariant forms. So when we restirct $f$ to a component of maximal level where the cover is given by an Artin--Schreier cover $f_i : C_i \to D_i$ with outgoing slopes $\ell_{1}, \dots, \ell_{z(C_i)}$, the curve $C_i$ of genus $h_i$ and $D_i$ having $\#H(D_i)^{\acute{e}t}$ many preimages of nodes where the map is étale, we see that the dimension of this Moduli is given by
    \begin{align*}
        \frac{2h_i}{p-1}+\#H(D_i)^{\acute{e}t}-1-\sum_{i=1}^{z(C_i)}\left \lfloor {\frac{\ell_i}{p(p-1)}} \right \rfloor
    \end{align*}
    Also at the top level, at some components $D_i$ of the target curve $D$ the map $f_i:f^{-1}( D_i) \to D_i$ might be étale. On such components, $f^{-1}(D_i)=C_i$ is just a disjoint union of $p$ copies of $\Pj$, each mapping by the identity to $D_i=\Pj$. So the Moduli of such a component is given by
    \begin{align*}
        \#H(D_i)-3.
    \end{align*}
    By $Mod^{AS}$ we denote the dimension of Moduli that is in the top level component. Let $V^{AS}_C$ be the set of top level components where the map is given by an Artin--Schreier cover and $V^{\acute{e}t}_D$ be the set of components of the target where the map is étale. By $Out^{AS}$ we denote the set of outgoing slopes from the top level and by $E_D^{\acute{e}t}$ we denote the set of edges between components at the maximal level of the target curve of slope $0$. By summing up we can now compute
    \begin{align*}
        Mod^{AS}=\sum_{v_i \in V_C^{AS}}\frac{ 2g(v_i)}{p-1} - \#V_C^{AS}- \sum_{\ell_i \in Out^{AS}} \left \lfloor {\frac{\ell_i}{p(p-1)}} \right \rfloor +2\#E_D^{\acute{e}t}- 3\#V^{\acute{e}t}_D
    \end{align*}
where $g(v_i)$ denotes the genus of the curve corresponding to $v_i$. 
    
    Next, we will compute the dimension of Moduli $Mod^{ex}$ coming from the components with exact differential forms. Let $C_i$ be an irreducible component on an intermediate level, $H(C_i)$ be the set of half-edges, so the number of markings and preimages of nodes under the normalization, $\ell_1, \dots ,\ell_{z(C_i)}$ be the outgoing and $\kappa_1, \dots,\kappa_{p(C_i)}$ be the incoming slopes. Note that at these components $\omega^{log}_f|_{C_i} \cong \Omega_{\mathbb{P}^1}^{(1-p)}$. Now Theorem~\ref{thm:exc_smooth} tells us that this Moduli problem is smooth of dimension
    \begin{align*}
    \#H(C_i)-4+\sum_{j=1}^{z(C_i)} \left \lfloor{\frac{\ell_j+(p-1)}{p}}\right \rfloor - \sum_{j=1}^{p(C_i)}\left \lceil{\frac{\kappa_j-(p-1)}{p}}\right \rceil.
    \end{align*}
    
    We know that the $\kappa_j, \ell_j$ are not congruent to $0$ modulo $p$ as we cannot integrate a differential form with a zero of order congruent to $-1$ modulo $p$. Therefore, the expression simplifies further to
     \begin{align*}
    2 \#H(C_i)-4+\sum_{j=1}^{z(C_i)} \left \lfloor{\frac{\ell_j}{p}}\right \rfloor - \sum_{j=1}^{p(C_i)}\left \lceil{\frac{\kappa_j}{p}}\right \rceil
    \end{align*}

    We denote by $In^{ex}$ the set of slopes that go in from the maximal level, so from an Artin--Schreier curve to the intermediate level, by $Out^{ex}$ those slopes that leave an intermediate level to the minimal level, so to a curve with an quasi-exact differential form, by $E_C^{ex}$ the set of edges between curves at intermediate levels and by $V_C^{ex}$ the set of components with an exact differential form. We can sum up the computed dimensions at each component and compute the dimension of Moduli at intermediate levels to be
    \begin{align*}
        Mod^{ex}= 2(\#In^{ex}+ Out^{ex})+3\#E_C^{ex} -4\#V_C^{ex}+\sum_{\ell_i \in Out^{ex}} \left \lfloor{\frac{\ell_i}{p}}\right \rfloor - \sum_{\kappa_i \in In^{ex}} \left \lceil{\frac{\kappa_i}{p}}\right \rceil 
    \end{align*}
    because two half-edges form one edge and $\left \lceil z \right \rceil - \left \lfloor z \right \rfloor=1$ for $z \notin \Z$.

    Finally, we compute the dimension $Mod^{qu-ex}$ of Moduli at quasi-exact components. So let $C_i$ be an irreducible component at the minimal level. Note that we do not have outgoing slopes from the minimal level. Again, by $\kappa_1, \dots, \kappa_{p(C_i)}$ we denote the incoming slopes at $C_i$. Then by the same techniques as for the case of exact differential forms, we compute the dimension of Moduli of quasi-exact differential forms at $C_i$ to be
    \begin{align*}
        \#H(C_i)+p(C_i)-3- \sum_{j=1}^{p(C_i)}\left \lceil{\frac{\kappa_j}{p}}\right \rceil
    \end{align*}
    where $H(C_i)$ is the set of half-edges, so points that correspond to edges or marked points at $C_i$. We denote by $In^{qu-ex}$ the set of incoming slopes into the minimal level and by $E_C^{qu-ex}$ be the set of edges of slope $0$ that connect components in the minimal level. 
    We now can sum up over all components of minimal level and compute $Mod^{qu.-ex.}$. As each half-edge either belongs to a marking or an edge and all markings are at the minimal level, this sum is:
    \begin{align*}
        Mod^{qu-ex}=2\#In^{qu-ex}+N+2\#E_C^{qu-ex}-3\#V_C^{qu-ex}-\sum_{\kappa_i \in In^{qu-ex}}\left \lceil{\frac{\kappa_i}{p}}\right \rceil
    \end{align*}
    with $V_C^{qu-ex}$ being the set of curves with a quasi-exact differential form.
    
    We can calculate that for a slope $\ell$ going out from a top level component we have
    \begin{align*}
        \left \lceil{\frac{\ell}{p}}\right \rceil+\left \lfloor{\frac{\ell}{p(p-1)}}\right \rfloor= \left \lceil{\frac{\ell p - \ell}{p(p-1)}}\right \rceil+\left \lfloor \frac{\ell}{p(p-1)} \right \rfloor = \frac{\ell}{p-1}
    \end{align*}
    In the last step we used that $\ell$ is divisible by $p-1$. 

    Furthermore, note that by the Riemann-Hurwitz formula we have
    \begin{align*}
        \sum_{\ell_i \in Out^{AS}} \ell_i= \sum_{v_i \in V_C^{AS}} \frac{2g(v_i)}{p-1}+2\#V_C^{AS}-\#Out^{AS}
    \end{align*}
    Using these equalities, we can sum up $Mod^{AS},Mod^{ex}$ and $Mod^{qu-ex}$ to compute the  dimension of the stratum of $U_0$ to be
    \begin{align*}
        dim(U_0)=N+3\#E_D-3\#V_D-\#E_D^{hor}-\#V_C^{ex}
    \end{align*}
    where $E_D$ is the set of nodes and $V_D$ the set of irreducible components of the target curve $D$. Note that $D$ is of genus $0$, so it holds 
    \begin{align*}
        \#E_D-\#V_D+1=0
    \end{align*}
    and the equation simplifies further to
    \begin{align*}
        dim(U_0)=N-3-\#E_D^{hor}-\#V_C^{ex}
    \end{align*}
    as desired.
\end{proof}

\begin{corollary}
    For $p=2$ and $g=0$, the irreducible components of the underlying stack of the special fiber of $\LHURphgnZ$ correspond to those graphs with only two levels and no horizontal edges.
\end{corollary}
\begin{proof}
    From Theorem~\ref{thm:equ-dim} we know that each irreducible component is of dimension $N-3$. According to the calculation we have just seen, this is the dimension exactly of those strata that are given by graphs with only two levels and no horizontal edges. 
\end{proof}
Recall that in order to prove Theorem~\ref{thm:log_hur_log_smooth}, we are left to show that

  \begin{align*}
    \gamma: U \to \Z_{(p)}[\N^{\#V_C^{ex}+\#E_D^{hor}}]
\end{align*}

is smooth.  
Note that is in fact enough to show that $\gamma$ is smooth at $x$ as we can shrink $U$ even further. We will use the Jacobi criteria to prove this. To simplify notation, set 
\begin{align*}
    r=\#V_C^{ex}+\#E_D^{hor}.
\end{align*}
After potentially shrinking $U$ further, $\gamma$ comes from a morphism
\begin{align*}
    \Z_{(p)}[t_1, \dots t_r] \to \Z_{(p)}[t_1, \dots,t_r, s_1, \dots,s_m]/(g_1, \dots g_v).
\end{align*}
Note that a priori it is not clear that $U$ is a local complete intersection, so we cannot assume that $v=r+m-N+3$.

We now consider the Jacobian
\begin{align*}
  {J_{g_1, \dots g_v} (x)} = (\frac{\partial g_i}{\partial s_j}(x)) _{i,j}
\end{align*}
We know that $x$ is a smooth point of the special fiber, so after reducing modulo $p$ the morphism is smooth at $x$ of relative dimension $N-3-r$. So the reduction of the Jacobian has a $(r+m-N+3) \times (r+m-N+3)$ minor $M$, such that the reduction modulo $p$ of the determinant of $M$ is non-zero. But then also the determinant of $M$ is non-zero and as the dimension of $U$ is $N-2$ we conclude that $\gamma$ is smooth at $x$ by the Jacobi criteria. This concludes the proof of Theorem~\ref{thm:log_hur_log_smooth} in the mixed characteristic case. 

We can show log smoothness of the equicharacteristic case by exactly the same technique. Still in the case of $p=2$ and $g=0$, the generic stratum of $\mathcal{LH}_{A,\Xi}^{\mathbb{F}_{2}}$ is $\mathcal{AS}cov^{\vec{e}, N_0}_h$ for $e_i=\xi_i+1$ and $N_0$ the cardinality of the image of marked unramified points. So the dimension of the generic stratum is 
\begin{align*}
    dim(\mathcal{LH}_{A,\Xi}^{\mathbb{F}_{2}})=N-3+\sum_{\xi_i \in \Xi} \frac{\xi_i +1}{2}
\end{align*}
For $x \in \mathcal{LH}_{A,\Xi}^{\mathbb{F}_{2}}$ corresponding to a cover $f: C \to D$ and a suitable étale neighbourhood $U$ of $x$, the chart at $x$ is given by
\begin{align*}
    \chi: U \to Spec(\mathbb{F}_2[\mathbb{N}^{\#V^{ex}_C+\#E^{hor}_D}]).
\end{align*}
In order to show log smoothness of $\mathcal{LH}_{A,\Xi}^{\mathbb{F}_{2}}$, one has now to show that this map is smooth and this follows from the following lemma which can be proven by a dimension counting argument as in the proof of Lemma~\ref{lem: dim_calc}:
\begin{lemma}
    The fiber $U_0$ of $\chi:U \to Spec(\mathbb{F}_2[\mathbb{N}^{\#V^{ex}_C+\#E^{hor}_D}])$ over the origin, determined by the vanishing of all coordinates is smooth of dimension
    \begin{align*}
        dim(U_0)= N-3+\sum_{\xi_i \in \Xi} \frac{\xi_i +1}{2} - (\#V^{ex}_C+\#E^{hor}_D)
    \end{align*}
\end{lemma}
\begin{corollary}
    For $p=2$ and $g=0$, the stack $\LHURphgnZ$ is the stack theoretic closure of the admissible cover stack in $\PLHURphgnZ$ and $\LHURphgnF$ is the stack theoretic closure of the locus of covers of smooth curves in $\PLHURphgnF$.
\end{corollary}
\begin{proof}
    In the proof of Theorem~\ref{thm:equ-dim} we saw that the underlying reduced stack of $\LHURphgnZ$ is the stack theoretic closure of the admissible cover stack in $\PLHURphgnZ$. So we are left to show that $\PLHURphgnZ$ has no associated components. 
    
    An associated component would intersect a stratum of the logarithmic stratification of $\LHURphgnZ$ but as the tangent space of the stratum is of the same dimension everywhere, the whole stratum would be contained in this associated component. But this is not possible as the stratum is smooth, so in particular reduced.

    The result of the equicharacteristic case follows by the same argument.
\end{proof}
\section{Example: Degree Two Covers of \texorpdfstring{$\Pj$}{P1} Ramified over Four Points} \label{sec:example}

In this section, we want to revisit the Moduli space of liftable Hurwitz covers of the projective line by elliptic curves that are ramified over four points over $\mathbb{Z}_{(2)}$ as it was originally described in \citep[Section 4]{abramovich1998stable}. We will see how the covers in the special fiber of the Moduli space that Abramovich and Oort explicitly constructed by computing stable models of Hurwitz covers in characteristic $0$ over $\Z_2$, can also directly be obtained as points of the special fiber of our logarithmic Hurwitz stack $\mathcal{LH}^{\Z_{(2)}}_{1,0,4,(2,2,2,2)}$. 

At first, Abramovich and Oort considered the elliptic cover of the projective line
\begin{align*}
    y^2=x(x-1)(x-\lambda)
\end{align*}
where modulo $2$ the marking $\lambda$ does not specialize to $0,1$ or $\infty$ and resolved singularities by a sequence of blowups in the special fiber. They ended up in a configuration where in the special fiber the cover is given by a supersingular elliptic curve mapping to $\Pjk$, but there is also a rational component of the source that maps inseparably to a rational component $\Pjk$ in the target. The four markings specialize to this inseparable component, and one sees that the point where this inseparable component is attached to the ramified point of the elliptic cover is given by $\sqrt{\lambda}$. Therefore, these covers give a onedimensional family in the special fiber of the Moduli space that is parametrized by $\lambda$. 

Already in this relatively simple case, one sees that various phenomena that only appear in positive characteristic do occur. With our tools, we can reveal more structure in this configuration, which can be seen in Figure~\ref{fig:irr_comp_lambda}.
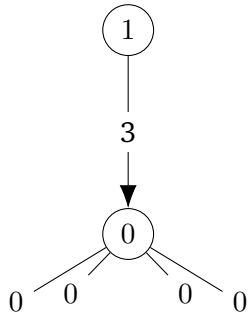
\begin{figure}[htb] 
   
\centering

 \begin{tikzpicture}
\node[circle,draw] (v0) at (0,0) {$1$};
\node[circle,draw, below = 2cm of v0] (v1) {$0$};

\node[below left = 0.4 cm and 1.0cm of v1] (m1) {$0$};
\node[below left = 0.4 cm of v1] (m2) {$0$};
\node[below right = 0.4 cm of v1] (m3) {$0$};
\node[below right = 0.4 cm and 1.0cm of v1] (m4) {$0$};

\draw[-{Latex[length=3mm]}] (v0) edge node[midway, fill=white]{3} (v1);

\draw [-] (v1) -- (m1);
\draw [-] (v1) -- (m2);
\draw [-] (v1) -- (m3);
\draw [-] (v1) -- (m4);

\end{tikzpicture}

\caption{The generic dual level graph $\Delta_{1;0}$ of the source curve of one irreducible component of the special fiber of our Moduli space. Note that when the four markings are labeled $0,1,\lambda$ and $\infty$, the elliptic curve is glued to the point $\sqrt{\lambda}$.}\label{fig:irr_comp_lambda}
\end{figure}
Covers in this stratum consist of an elliptic curve that maps to the projective line separably and such that the cover is ramified only over one point, so the curve is supersingular and given by the equation
\begin{align*}
    y^2+y=x^3.
\end{align*}
One can check that the bivariant differential form at this component, which is the trace form, has a zero of order $4$, which is compatible with the graph as we count the order of zeroes and poles logarithmically.

On the component where the map is inseparable, the condition is that the bivariant form with simple zeroes  at $0,1,\lambda, \infty$ and a pole of order $2$ at $\mu$ that is the attaching point, is quasi-exact. To check this, we compute
\begin{align*}
    \mathfrak{tc} \left( \frac{y(y-1)(y-\lambda)}{(y-\mu)^2}  \frac{dy}{dx}\right) = \frac{1}{(y-\mu)} \mathfrak{tc}\left((y^3-(1+\lambda)y^2 - \lambda y) \frac{dy}{dx} \right) = \frac{y-\sqrt\lambda}{y-\mu}
\end{align*}
This is in $k^\times$, so the form is quasi-exact, iff $\mu = \sqrt{\lambda}$ as expected.

Here we also see the difference between $\PLHURphgnZ$ and $\LHURphgnZ$: Without the exactness and quasi-exactness condition, one would not impose any relation between $\mu$ and $\lambda$. Compared to $\mathcal{LH}^{\Z_{(2)}}_{1,0,4,(2,2,2,2)}$ the Moduli space $\mathcal{PLH}^{\Z_{(2)}}_{1,0,4,(2,2,2,2)}$ has an additional two-dimensional irreducible component that is generically supported in characteristic $2$ and correspond to the configuration of $\Delta_{1;0}$ without any requirements on the configuration of markings and nodes. This irreducible component intersects $\mathcal{LH}^{\Z_{(2)}}_{1,0,4,(2,2,2,2)}$ in those points in the strata of $\Delta_{1;0}$ where the relation $\sqrt{\lambda}=\mu$ is satisfied. From this, one sees in particular that in contrast to $\LHURphgnZ$, the stack $\PLHURphgnZ$ is not log smooth.

The next case Abramovich and Oort consider is the case where $\lambda$ reduces modulo $2$ to either $0,1$ or $\infty$, but the $j$-invariant does not reduce to $0$. In this case, one obtains the configuration of Figure~\ref{fig:irr_comp_j}.
\begin{figure}[htb] 
   
\centering

 \begin{tikzpicture}
\node[circle,draw] (v0) at (0,0) {$1$};
\node[circle,draw, below left= 2cm and 1cm of v0] (v1) {$0$};
\node[circle,draw, below right= 2cm and 1cmof v0] (v2) {$0$};

\node[below left = 0.4 cm of v1] (m1) {$0$};
\node[below right = 0.4 cm of v1] (m2) {$0$};
\node[below left = 0.4 cm of v2] (m3) {$0$};
\node[below right = 0.4 cm of v2] (m4) {$0$};

\draw[-{Latex[length=3mm]}] (v0) edge node[midway, fill=white]{1} (v1);
\draw[-{Latex[length=3mm]}] (v0) edge node[midway, fill=white]{1} (v2);

\draw [-] (v1) -- (m1);
\draw [-] (v1) -- (m2);
\draw [-] (v2) -- (m3);
\draw [-] (v2) -- (m4);

\end{tikzpicture}

\caption{The dual level graph $\Delta_{1;0,0}$ of the target curve of three strata of the special fiber of our Moduli space.}\label{fig:irr_comp_j}
\end{figure}
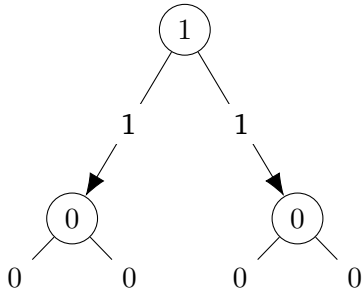
These component are parametrized by the $j$-invariant of the curve at the separable level, where we have seen that the cover is given by the Artin--Schreier equation
\begin{align*}
    y^2-y= \frac{1}{x} +\frac{1}{x-a}
\end{align*}
This gives us a onedimensional family. Furthermore, it is easy to check that the bivariant form $y(y-1)\frac{dy}{dx}$ is quasi-exact, and as there are only three special points at components of the minimal level, there is no Moduli coming from these components. This configuration gives us three strata in the special fiber of $\mathcal{LH}^{\Z_{(2)}}_{1,0,4,(2,2,2,2)}$ that correspond to different order of the markings.

Note that the configurations that we have considered so far, had two levels and gave us smooth, one-dimensional strata. This fits in what we have seen in the last section, as also the admissible cover stack parametrizing double covers of the projective line by an elliptic curve ramified over four points in characteristic $0$ is onedimensional. 

Coming from the case of the strata described by $\Delta_{1;0,0}$, Abramovich and Oort describe two additional configurations that can be seen in Figure~\ref{fig:boundary}.
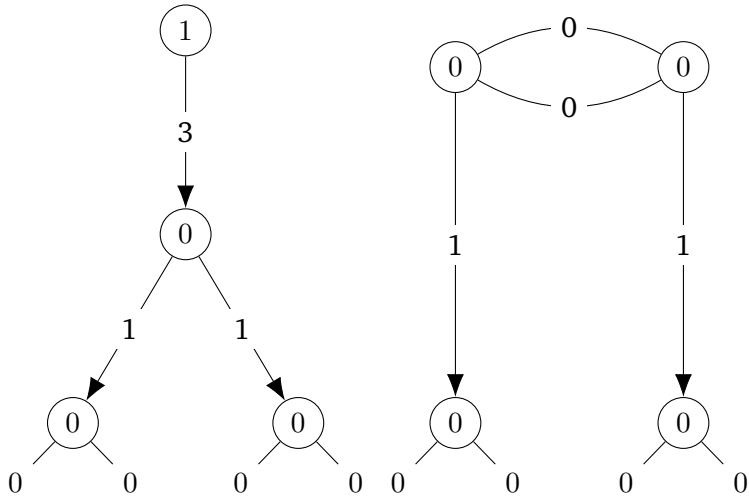
\begin{figure}[htb] 
\centering
   \begin{tabular}{ c c }

\begin{tikzpicture}
\node[circle,draw] (v0) at (0,0) {$1$};
\node[circle,draw, below = 2cm of v0] (v1)  {$0$};
\node[circle,draw, below left= 2cm and 1cm of v1]  (v2)  {$0$};
\node[circle,draw, below right =2cm and 1cm of v1] (v3)  {$0$};

\node[below left = 0.4 cm of v2] (m1) {$0$};
\node[below right = 0.4 cm of v2] (m2) {$0$};
\node[below left = 0.4 cm of v3] (m3) {$0$};
\node[below right = 0.4 cm of v3] (m4) {$0$};

\draw[-{Latex[length=3mm]}] (v0) edge node[midway, fill=white]{3} (v1);
\draw[-{Latex[length=3mm]}] (v1) edge node[midway, fill=white]{1} (v2);
\draw[-{Latex[length=3mm]}] (v1) edge node[midway, fill=white]{1} (v3);

\draw [-] (v2) -- (m1);
\draw [-] (v2) -- (m2);
\draw [-] (v3) -- (m3);
\draw [-] (v3) -- (m4);
\end{tikzpicture}

 \begin{tikzpicture}
\node[circle,draw] (v0) at (0,0) {$0$};
\node[circle,draw] (v1) at (3,0) {$0$};
\node[circle,draw, below =4cm of v0] (v2)  {$0$};
\node[circle,draw, below =4cm of v1] (v3) {$0$};

\node[below left = 0.4 cm of v2] (m1) {$0$};
\node[below right = 0.4 cm of v2] (m2) {$0$};
\node[below left = 0.4 cm of v3] (m3) {$0$};
\node[below right = 0.4 cm of v3] (m4) {$0$};

\draw[-] (v0) edge[bend right] node[midway, fill=white]{0} (v1);
\draw[-] (v0) edge[bend left] node[midway, fill=white]{0} (v1);
\draw[-{Latex[length=3mm]}] (v0) edge node[midway, fill=white]{1} (v2);
\draw[-{Latex[length=3mm]}] (v1) edge node[midway, fill=white]{1} (v3);

\draw [-] (v2) -- (m1);
\draw [-] (v2) -- (m2);
\draw [-] (v3) -- (m3);
\draw [-] (v3) -- (m4);

\end{tikzpicture} 

\end{tabular}

\caption{The two graphs $\Delta_{1;0;0,0}$ and $\Delta_{0,0;0;0}$ corresponding to strata in the boundary of the strata described by $\Delta_{1;0,0}$.}\label{fig:boundary}
\end{figure}
When in the last case $j$ specializes to $0$ one obtains the configuration of graph $\Delta_{1;0;0,0}$ and when $j$ specializes to $\infty$, then one obtains $\Delta_{1;0,0}$ from Figure~\ref{fig:boundary}. Note that depending on the order of the markings, there are three strata for both of the given dual level graphs.

If a cover is in the stratum of $\Delta_{1;0;0,0}$, as in the stratum $\Delta_{1;0,0}$, at the maximum level the cover is given by a supersingular elliptic curve at the maximal level and at the intermediate level one verifies that the relative differential form $\frac{y^2(y-1)^2}{(y-\infty)^2} \frac{dy}{dx}$ is exact. As on each rational component there are just three special points, the strata is just a point. The logarithmic structure at such a point is a free monoid of rank $2$, generated by the minimal level $log(p)$ and a parameter for the intermediate level where the rational component with an exact differential form is located.

In the strata of $\Delta_{0,0;0,0}$, note that both bent edges correspond to nodes that map to the same node of the target curve, so the map is étale there. At both components of the maximal level the cover is given by
\begin{align*}
    y^2+y=x
\end{align*}
where the curves are glued together at two points, that map to the same point in the target. This stratum also consists just of three points corresponding to the order of the markings. The log structure here is freely generated by the minimal level $log(p)$ and one smoothing parameter for the node of the target curve that is the image of the horizontal edges. Note that the bent edges of the source curve do not get independent smoothing parameters, as they map to the same node in the target.

The strata of $\Delta_{1;0;0,0}$ are also in the boundary of the respective component of $\Delta_{1;0}$ and is obtained when two zeroes of the quasi-exact bivariant form collide. 

To summarize, the special fiber in characteristic $2$ of $\mathcal{LH}^{\Z_{(2)}}_{1,0,4,(2,2,2,2)}$ is a log curve that can be constructed by gluing three copies of $\overline{\mathcal{M}}_{1,1}$ with coordinate being the $j$-invariant to a projective line with coordinate $\lambda$. Each such copy of $\overline{\mathcal{M}}_{1,1}$ is glued to the projective line by identifying the point $j=0$ with $\lambda=0,1$ or $\infty$ respectively. Each copy of $\overline{\mathcal{M}}_{1,1}$ has a marking at $j= \infty$ as the minimal monoids of the strata of $\Delta_{0,0:0,0}$ have an additional smoothing parameter there.
\bibliographystyle{alpha}
\bibliography{bib}

\end{document}